\documentclass[reqno]{amsart}
\usepackage{setspace,tikz,xcolor,mathrsfs,listings,multicol}
\usepackage{amssymb}
\usepackage{rotating}
\usepackage[vcentermath]{youngtab}
\usepackage{fullpage}
\usepackage{enumerate}
\usepackage[all,cmtip]{xy}
\usetikzlibrary{arrows,matrix}

\usepackage[colorlinks=true, pdfstartview=FitV, linkcolor=blue, citecolor=blue, urlcolor=blue]{hyperref}

\newcommand{\g}{\mathfrak{g}}
\newcommand{\hh}{\mathfrak{h}}

\newcommand{\inner}[2]{\left\langle #1, #2 \right\rangle}
\newcommand{\iso}{\cong}

\newcommand{\sgn}{\operatorname{sgn}}

\newcommand{\rootarrow}[1]{\xrightarrow[\hspace{30pt}]{#1}}

\newcommand{\pr}{\mathrm{pr}}
\newcommand{\inc}{\mathrm{in}}

\newcommand{\Al}{\mathcal{A}} 
\newcommand{\h}[1] { h^{J}_{#1} } 
\newcommand{\IndSet}{\mathcal{K}}  
\newcommand{\ind}{k}  

\DeclareMathOperator{\wt}{wt} 

\newcommand{\ZZ}{\mathbb{Z}}
\newcommand{\QQ}{\mathbb{Q}}
\newcommand{\RR}{\mathbb{R}}

\usepackage{listings}
\lstdefinelanguage{Sage}[]{Python}
{morekeywords={False,sage,True},sensitive=true}
\definecolor{dblackcolor}{rgb}{0.0,0.0,0.0}
\definecolor{dbluecolor}{rgb}{0.01,0.02,0.7}
\definecolor{dgreencolor}{rgb}{0.2,0.4,0.0}
\definecolor{dgraycolor}{rgb}{0.30,0.3,0.30}

\usepackage{xparse}

\makeatletter
\protected\def\specialmergetwolists{%
  \begingroup
  \@ifstar{\def\cnta{1}\@specialmergetwolists}
    {\def\cnta{0}\@specialmergetwolists}%
}
\def\@specialmergetwolists#1#2#3#4{%
  \def\tempa##1##2{%
    \edef##2{%
      \ifnum\cnta=\@ne\else\expandafter\@firstoftwo\fi
      \unexpanded\expandafter{##1}%
    }%
  }%
  \tempa{#2}\tempb\tempa{#3}\tempa
  \def\cnta{0}\def#4{}%
  \foreach \x in \tempb{%
    \xdef\cnta{\the\numexpr\cnta+1}%
    \gdef\cntb{0}%
    \foreach \y in \tempa{%
      \xdef\cntb{\the\numexpr\cntb+1}%
      \ifnum\cntb=\cnta\relax
        \xdef#4{#4\ifx#4\empty\else,\fi\x#1\y}%
        \breakforeach
      \fi
    }%
  }%
  \endgroup
}
\makeatother

\usepackage{xparse}
\DeclareDocumentCommand\rpp{ m m g }{
	\foreach \x [count=\s from 1] in {#1}{
	        {\ifnum\s=1
	                \draw (0,-\s)--(\x,-\s);
	                \fi}
	   \draw (0,-\s-1) to (\x,-\s-1);
	   \foreach \y in {0, ..., \x} {\draw (\y,-\s)--(\y,-\s-1);}
	}
	\specialmergetwolists{/}{#1}{#2}\ziplist
	\foreach \x/\y [count=\yi from 1] in \ziplist{
	    \node[anchor=west,font=\small] at (\x,-\yi - .5) {$\y$};
	}
	\IfValueT {#3}
	{\foreach \z [count=\zi from 1] in {#3} {\node[anchor=east,font=\small] at (0,-\zi - .5) {$\z$};}}
	{}
}

\definecolor{darkred}{rgb}{0.7,0,0} 
\newcommand{\defn}[1]{{\color{darkred}\emph{#1}}} 

\usepackage[colorinlistoftodos]{todonotes}

\theoremstyle{plain}
\newtheorem{thm}{Theorem}[section]
\newtheorem{lemma}[thm]{Lemma}

\newtheorem{prop}[thm]{Proposition}
\newtheorem{cor}[thm]{Corollary}
\theoremstyle{definition}
\newtheorem{dfn}[thm]{Definition}
\newtheorem{ex}[thm]{Example}
\newtheorem{remark}[thm]{Remark}

\numberwithin{equation}{section}

\begin{document}
\title{Alcove path model for $B(\infty)$}

\author{Arthur Lubovsky}
\address[A.~Lubovsky]{Department of Mathematics, University of New York at Albany, Albany, NY 12222}
\email{alubovsky@albany.edu}

\author{Travis Scrimshaw}
\address[T.~Scrimshaw]{School of Mathematics, University of Minnesota, Minneapolis, MN 55455} 
\curraddr{School of Mathematics and Physics, University of Queensland, St. Lucia, QLD 4072, Australia}
\email{tcscrims@gmail.com}
\urladdr{https://sites.google.com/view/tscrim/home}


\thanks{TS was partially supported by the National Science Foundation RTG grant NSF/DMS-1148634.}

\begin{abstract}
We construct a model for $B(\infty)$ using the alcove path model of Lenart and Postnikov.
We show that the continuous limit of our model recovers a dual version of the Littelmann path model for $B(\infty)$ given by Li and Zhang.
Furthermore, we consider the dual version of the alcove path model and obtain analogous results for the dual model, where the continuous limit gives the Li and Zhang model.
\end{abstract}

\maketitle

\section{Introduction}
\label{sec:introduction}

The theory of Kashiwara's crystal bases~\cite{K90,K91} has been shown to have deep connections with numerous areas of geometry and combinatorics, well-beyond its origin in representation theory and mathematical physics.
A crystal basis is a particularly nice basis for certain representations of a quantum group $U_q(\g)$ in the limit $q \to 0$, or crystal limit.
In particular, for a symmetrizable Kac--Moody algebra $\g$, the integrable highest weight modules $V(\lambda)$, so $\lambda$ is a dominant integral weight, were shown by Kashiwara to admit crystal bases $B(\lambda)$.
Moreover, Kashiwara has shown that the lower half of the quantum group $U_q^-(\g)$ admits a crystal basis $B(\infty)$.

Roughly speaking, the algebraic action of $U_q(\g)$ gets transformed into a combinatorial action on the bases in the $q \to 0$ limit.
While Kashiwara's grand loop argument showed the existence of the crystal bases $B(\lambda)$, it did not give an explicit (combinatorial) description.
Thus the problem was to determine a combinatorial model for $B(\lambda)$.
This was first done for $\g$ of type $A_n$, $B_n$, $C_n$, and $D_n$ in~\cite{KN94} and $G_2$ in~\cite{KM94} by using tableaux.
A uniform model (for all symmetrizable types) for crystals using piecewise-linear paths in the weight space was constructed in~\cite{L95,L95-2}, which is now known as the Littelmann path model.
A special case of the Littelmann path model includes Lakshmibai--Seshadri (LS) paths, where the combinatorial definition was given by Stembridge~\cite{Stembridge02}.

Both of these models arose from examining a particular aspect of the representation theory of $\g$ and the related combinatorics or geometry.
There are numerous (but not necessarily uniform) models for $B(\lambda)$ that have been constructed from geometric objects such as quiver varieties~\cite{KS97,Saito02,Savage05} and MV polytopes~\cite{BKT14,Kamnitzer07,MT14,TW16}.
Another uniform model for crystals came from the study of ($t$-analogs of) $q$-characters~\cite{K03II,Nakajima03II,Nakajima03,Nakajima04}, which is now known as Nakajima monomials.
Additionally, some models for crystals have also arisen from mathematical physics, in particular, solvable lattice models~\cite{KKMMNN91,KKMMNN92} (the Kyoto path model) and Kirillov--Reshetikhin modules~\cite{SalisburyS15,SalisburyS16,SalisburyS16II,SalisburyS15II,S06,SchillingS15} (the rigged configuration model).

Many of these models are known to have extensions to $B(\infty)$.
Some authors have used the direct limit construction of Kashiwara~\cite{K02} to extend a particular crystal model for $B(\lambda)$ to $B(\infty)$. Examples include the tableaux model~\cite{Cliff98,HL08,HL12} and rigged configurations~\cite{SalisburyS15,SalisburyS16,SalisburyS16II,SalisburyS15II}, where the model reflects the naturality of the inclusion of $B(\lambda) \to B(\mu)$ for $\lambda \leq \mu$.
In contrast, other authors have used other characterizations of $B(\infty)$ to construct their extensions, such as the polyhedral realization~\cite{NZ97} (which has a $B(\lambda)$ version~\cite{Hoshino13,Hoshino05,HN05,Nakashima99}), Nakajima monomials~\cite{KKS07}, and Littelmann paths~\cite{LZ11}.

The model we will be focusing on is a discrete version of the Littelmann path model known as the alcove path model that was given for $B(\lambda)$ in~\cite{LP07,LP08}.
The alcove path model in finite types is related to LS galleries and Mirkovi\'c--Vilonen (MV) cycles~\cite{GL05} and the equivariant $K$-theory of the generalized flag variety~\cite{LP07}.
Moreover, the alcove path model can be described in terms of certain saturated chains in the (strong) Bruhat poset.
While the Littelmann path model came first, it is perhaps more proper to consider the Littelmann path model as the continuous limit of the alcove path model.
Moreover, the alcove path model carries with it more information, specifically the order in which the hyperplanes are crossed, allowing a non-recursive description of the elements in full generality.

The primary goal of this paper is to construct a model for $B(\infty)$ using the alcove path model.
Our approach is to use the direct limit construction of Kashiwara restricted to $\{B(k\rho)\}_{k=0}^{\infty}$,\footnote{We omit the weight shifting crystals $T_{-k\rho}$ for simplicity of our exposition in the introduction.} where the inclusions $\psi_{k\rho, k'\rho} \colon B(k\rho) \to B(k'\rho)$, for $k' \geq k$, are easy to compute.
In order to do so, we define the concatenation of a $\lambda$-chain and a $\mu$-chain (see Definition~\ref{def:concatenation}).
We then complete our proof by using the fact that for every $b \in B(\infty)$, there exists a $k \gg 1$ such that $b$ and $f_i b$, for all $i$, is not in the kernel of the natural projection onto $B(k\rho)$.
Next, the continuous limit of the alcove path model for $B(\lambda)$ to the Littelmann path model for $B(-\lambda)$ is given explicitly by~\cite[Thm.~9.4]{LP08} as a ``dual'' crystal isomorphism $\varpi_{\lambda}$.
We extend $\varpi_{\lambda}$ to an explicit crystal isomorphism beween the alcove path model and Littelmann path model for $B(\infty)$.

One of the strengths of the alcove path model for $B(\lambda)$ is that the elements in the crystal are given non-recursively; in particular, they are not constructed by applying the crystal operators to the highest weight element.
We retain the notion of an admissible sequence when we consider $B(\infty)$.
Thus, the check whether an element is in $B(\infty)$ is a matter of checking if the foldings in an alcove walk correspond to a saturated chain in the Bruhat order.
Hence, we obtain the first model for $B(\infty)$ that has a non-recursive description of its elements in all symmetrizable types.
Previously, if the model had a non-recursive definition, it was either type-specific (for example~\cite{Cliff98,HL08,Kamnitzer07,KKM94}) or described as the closure under the crystal operators (for example~\cite{KKS07,LZ11,SalisburyS15,SalisburyS15II}).

In order to construct the continuous limit of the alcove path model for $B(\infty)$ in analogy to~\cite[Thm.~9.4]{LP07}, we need to construct a Littelmann path model for the contragrediant dual of $B(\infty)$.
We note that we can construct the contragrediant dual crystal explicitly in terms of (finite length) Littelmann paths by reversing a path and changing the starting point.
Using this as a base, we construct a new model that no longer starts at the origin, unlike the usual Littelmann path model (or the natural model for $B(-\infty)$ as the direct limit of $\{B(-k\rho)\}_{k=0}^{\infty}$), but, roughly speaking, ``at infinity.''
We show that the map described by~\cite{LP07} extends to the $B(\infty)$ case and is a dual isomorphism between the two models.
Moreover, in an effort to avoid using the dual Littelmann path model, we are led to construct a dual alcove path model that is essentially given by reversing the alcove path, mimicking the contragrediant dual construction on Littelmann paths (see Theorem~\ref{thm:dual_infinity_alcove_to_path}).
We then show that dual alcove path model is dual isomorphic to the usual Littelmann path model.

This paper is organized as follows.
In Section~\ref{sec:background}, we give the necessary background on crystals and the alcove path model.
In Section~\ref{sec:alcove_infinity}, we describe our alcove path model for $B(\infty)$.
In Section~\ref{sec:results}, we prove our main results.
In Section~\ref{sec:continuous_limit infinite}, we construct an isomorphism between our model and the (dual) Littelmann path model.

\section{Background}
\label{sec:background}

In this section, we give a background on general crystals, the crystal $B(\infty)$, and the alcove path model.

\subsection{Crystals}
\label{sec:crystal_bg}

Let $\g$ be a symmetrizable Kac--Moody algebra with index set $I$, generalized Cartan matrix $A = (A_{ij})_{i,j\in I}$, weight lattice $P$, root lattice $Q$, fundamental weights $\{\Lambda_i \mid i \in I\}$, simple roots $\{\alpha_i \mid i\in I\}$, simple coroots $\{\alpha_i^{\vee} \mid i\in I\}$, and Weyl group $W$.
Let $U_q(\g)$ be the corresponding Drinfel'd--Jimbo quantum group~\cite{Drinfeld85,Jimbo85}.
Let $\hh^*_{\RR} := \RR \otimes_{\ZZ} P$ and $\hh_{\RR} := \RR \otimes_{\ZZ} P^{\vee}$ be the corresponding dual space, where $P^{\vee}$ is the coweight lattice.
We also denote the canonical pairing $\inner{\cdot}{\cdot}
\colon \hh_{\RR}^* \times \hh_{\RR} \to \RR$ given by $\inner{\alpha_i}{\alpha_j^{\vee}} = A_{ij}$.
Let $\Phi^+$ denote positive roots, $P^+$ denote the dominant weights, and $\rho = \sum_{i \in I} \Lambda_i$.
For a root $\alpha$, the corresponding coroot is $\alpha^{\vee} := 2 \alpha / \langle \alpha, \phi(\alpha) \rangle$, where $\phi \colon \hh_{\RR}^* \to \hh_{\RR}$ is the $\RR$-linear isomorphism given by $\phi(\alpha_i) = \alpha_i^{\vee}$.

An \defn{abstract $U_q(\g)$-crystal} is a nonempty set $B$ together with maps
\begin{align*}
e_i, f_i & \colon B \to B\sqcup\{0\},
\\ \varepsilon_i, \varphi_i & \colon B \to \ZZ \sqcup \{-\infty\}, \\
\wt & \colon B \to P, 
\end{align*}
which satisfy the properties
\begin{enumerate}
\item $\varphi_i(b) = \varepsilon_i(b) + \inner{\wt(b)}{\alpha_i^{\vee}}$ for all $i \in I$,
\item if $b\in B$ satisfies $e_ib \neq 0$, then
\begin{enumerate}
\item $\varepsilon_i(e_ib) = \varepsilon_i(b) - 1$,
\item $\varphi_i(e_ib) = \varphi_i(b) + 1$,
\item $\wt(e_ib) = \wt(b) + \alpha_i$,
\end{enumerate}
\item if $b\in B$ satisfies $f_ib \neq 0$, then
\begin{enumerate}
\item $\varepsilon_i(f_ib) = \varepsilon_i(b) + 1$,
\item $\varphi_i(f_ib) = \varphi_i(b) - 1$,
\item $\wt(f_ib) = \wt(b) - \alpha_i$,
\end{enumerate}
\item $f_ib = b^{\prime}$ if and only if $b = e_i b^{\prime}$ for $b,b^{\prime} \in B$ and $i\in I$,
\item if $\varphi_i(b) = -\infty$ for $b\in B$, then $e_i b = f_i b =0$.
\end{enumerate}
The maps $e_i$ and $f_i$, for $i \in I$, are called the \defn{crystal operators} or \defn{Kashiwara operators}. We refer the reader to \cite{HK02,K91} for details.

We call an abstract $U_q(\g)$-crystal \defn{upper regular} if
\[
\varepsilon_i(b) = \max\{ k \in\ZZ_{\ge0} \mid e_i^k b \neq 0 \}
\]
for all $b\in B$. Likewise, an abstract $U_q(\g)$-crystal is \defn{lower regular} if
\[
\varphi_i(b) = \max\{ k \in \ZZ_{\ge0} \mid f_i^kb \neq 0\}
\]
for all $b\in B$. When $B$ is both upper regular and lower regular, then we say $B$ is \defn{regular}. For $B$ a regular crystal, we can express an entire $i$-string through an element $b \in B$ diagrammatically by
\[
e_i^{\varepsilon_i(b)}b \overset{i}{\longrightarrow}
\cdots \overset{i}{\longrightarrow}
e_i^2 b \overset{i}{\longrightarrow}
e_ib \overset{i}{\longrightarrow}
b \overset{i}{\longrightarrow}
f_ib \overset{i}{\longrightarrow}
f_i^2 b \overset{i}{\longrightarrow}
\cdots \overset{i}{\longrightarrow}
f_i^{\varphi_i(b)}b.
\]

An abstract $U_q(\g)$-crystal is called \defn{highest weight} if there exists an element $u \in B$ such that $e_i u = 0$ for all $i \in I$ and there exists a finite sequence $(i_1, i_2, \dotsc, i_{\ell})$ such that $b = f_{i_1} f_{i_2} \cdots f_{i_{\ell}} u$ for all $b \in B$. The element $u$ is called the \defn{highest weight element}.

Let $B_1$ and $B_2$ be two abstract $U_q(\g)$-crystals.  A \defn{crystal morphism} $\psi\colon B_1 \to B_2$ is a map $B_1\sqcup\{0\} \to B_2 \sqcup \{0\}$ such that
\begin{enumerate}
\item $\psi(0) = 0$;
\item if $b \in B_1$ and $\psi(b) \in B_2$, then $\wt(\psi(b)) = \wt(b)$, $\varepsilon_i(\psi(b)) = \varepsilon_i(b)$, and $\varphi_i(\psi(b)) = \varphi_i(b)$;
\item for $b \in B_1$, we have $\psi(e_i b) = e_i \psi(b)$ provided $\psi(e_ib) \neq 0$ and $e_i\psi(b) \neq 0$;
\item for $b\in B_1$, we have $\psi(f_i b) = f_i \psi(b)$ provided $\psi(f_ib) \neq 0$ and $f_i\psi(b) \neq 0$.
\end{enumerate}
A morphism $\psi$ is called \defn{strict} if $\psi$ commutes with $e_i$ and $f_i$ for all $i \in I$.  Moreover, a morphism $\psi\colon B_1 \to B_2$ is called an \defn{embedding} or \defn{isomorphism} if the induced map $B_1 \sqcup\{0\} \to B_2 \sqcup \{0\}$ is injective or bijective, respectively. If there exists an isomorphism between $B_1$ and $B_2$, say they are isomorphic and write $B_1 \iso B_2$.

The tensor product $B_2 \otimes B_1$ is the crystal whose set is the Cartesian product $B_2\times B_1$ and the crystal structure given by
\begin{align*}
e_i(b_2 \otimes b_1) &= \begin{cases}
e_i b_2 \otimes b_1 & \text{if } \varepsilon_i(b_2) > \varphi_i(b_1), \\
b_2 \otimes e_i b_1 & \text{if } \varepsilon_i(b_2) \le \varphi_i(b_1),
\end{cases} \\
f_i(b_2 \otimes b_1) &= \begin{cases}
f_i b_2 \otimes b_1 & \text{if } \varepsilon_i(b_2) \ge \varphi_i(b_1), \\
b_2 \otimes f_i b_1 & \text{if } \varepsilon_i(b_2) < \varphi_i(b_1),
\end{cases} \\ 
\varepsilon_i(b_2 \otimes b_1) &= \max\big( \varepsilon_i(b_1), \varepsilon_i(b_2) - \inner{\alpha_i^{\vee}}{\wt(b_1)} \bigr), \\
\varphi_i(b_2 \otimes b_1) &= \max\big( \varphi_i(b_2), \varphi_i(b_1) + \inner{\alpha_i^{\vee}}{\wt(b_2)} \bigr), \\
\wt(b_2 \otimes b_1) &= \wt(b_2) + \wt(b_1).
\end{align*}

\begin{remark}
\label{rem:tensor_order}
Our convention for tensor products is opposite the convention given by Kashiwara in~\cite{K91}.
\end{remark}


We say an abstract $U_q(\g)$-crystal is simply a \defn{$U_q(\g)$-crystal} if it is crystal isomorphic to the crystal basis of a $U_q(\g)$-module.

The highest weight $U_q(\g)$-module $V(\lambda)$ for $\lambda \in P^+$ has a crystal basis~\cite{K90,K91}. The corresponding (abstract) $U_q(\g)$-crystal is denoted by $B(\lambda)$, and we denote the highest weight element by $u_{\lambda}$. Moreover, the negative half of the quantum group $U_q^-(\g)$ admits a crystal basis denoted by $B(\infty)$, and we denote the highest weight element by $u_{\infty}$. Note that $B(\lambda)$ is a regular $U_q(\g)$-crystal, but $B(\infty)$ is only upper regular.

Consider a directed system of abstract $U_q(\g)$-crystals $\{B_j\}_{j\in J}$ with crystal morphisms $\psi_{k,j}\colon B_j \to B_k$ for $j \leq k$ (with $\psi_{j,j}$ being the identity map on $B_j$) such that $\psi_{k,j} \psi_{j,i} = \psi_{k,i}$ for $i \leq j \leq k$.
Let $\vec{B} = \varinjlim_{j \in J} B_j$ be the direct limit of this system, and let $\psi_{(j)} \colon B_j \to \vec{B}$.
Then Kashiwara showed in~\cite{K02} that $\vec{B}$ has a crystal structure induced from the crystals $\{B_j\}_{j \in J}$; in other words, direct limits exist in the category of abstract $U_q(\g)$-crystals.
Specifically, for $\vec{b} \in \vec{B}$ and $i\in I$, define $e_i\vec{b}$ to be $\psi_{(j)}(e_i b_j)$ if there exists $b_j \in B_j$ such that $\psi_{(j)}(b_j) = \vec{b}$ and $e_i(b_j) \neq 0$, otherwise set $e_i\vec{b} = 0$.
Note that this definition does not depend on the choice of $b_j$.
The definition of $f_i\vec{b}$ is similar.
Moreover, the functions $\wt$, $\varepsilon_i$, and $\varphi_i$ on $B_j$ extend to functions on $\vec{B}$.

\begin{dfn}
\label{def:T_crystal}
For a weight $\lambda$, let $T_\lambda = \{t_\lambda\}$ be the abstract $U_q(\g)$-crystal with operations defined by
\begin{gather*}
e_i t_\lambda = f_i t_\lambda = 0,
\\ \varepsilon_i(t_\lambda) = \varphi_i(t_\lambda) = -\infty,
\\ \wt(t_\lambda) = \lambda,
\end{gather*}
for any $i \in I$.
\end{dfn}

Consider an abstract $U_q(\g)$-crystal $B$, then the tensor product $T_{\lambda} \otimes B$ has the same crystal graph as $B$ (but the weight, $\varepsilon_i$, and $\varphi_i$ have changed).  Next, we recall from~\cite{K02} that the map
\[
\psi_{\lambda+\mu,\lambda}\colon  T_{-\lambda}\otimes B(\lambda) \lhook\joinrel\longrightarrow T_{-\lambda-\mu} \otimes B(\lambda+\mu)
\]
which sends $t_{-\lambda}\otimes u_\lambda  \mapsto t_{-\lambda-\mu}\otimes u_{\lambda+\mu}$ is a crystal embedding, and this morphism commutes with $e_i$ for each $i \in I$.  Moreover, for any $\lambda, \mu, \xi \in P^+$, the diagram
\begin{equation}
\label{eq:directed_system}
\begin{tikzpicture}[xscale=5,yscale=1.5,baseline=0.6cm]
\node (1) at (0,1) {$T_{-\lambda} \otimes B(\lambda)$};
\node (2) at (1,1) {$T_{-\lambda-\mu}\otimes B(\lambda+\mu)$};
\node (3) at (1,-0.5) {$T_{-\lambda-\mu-\xi}\otimes B(\lambda+\mu+\xi)$};
\path[->,font=\scriptsize]
 (1) edge node[above]{$\psi_{\lambda+\mu,\lambda}$} (2)
 (1) edge node[anchor=north east]{$\psi_{\lambda+\mu+\xi,\lambda}$} (3)
 (2) edge node[right]{$\psi_{\lambda+\mu+\xi,\lambda+\mu}$} (3);
\end{tikzpicture}
\end{equation}
commutes. Furthermore, if we order $P^+$ by $\mu \leq \lambda$ if and only if $\inner{\lambda-\mu}{\alpha_i^{\vee}} \geq 0$ for all $i \in I$, the set $\{T_{-\lambda}\otimes B(\lambda)\}_{\lambda\in P^+}$ is a directed system.  

\begin{thm}[\cite{K02}]
\label{thm:direct_limit_construction}
We have
\[
B(\infty) = \varinjlim_{\lambda\in P^+} T_{-\lambda}\otimes B(\lambda).
\]
\end{thm}

From Theorem~\ref{thm:direct_limit_construction}, we have that for any $\lambda \in P^+$, there exists a natural projection $p_{\lambda} \colon B(\infty) \to T_{-\lambda} \otimes B(\lambda)$ and inclusion $i_{\lambda} \colon T_{-\lambda} \otimes B(\lambda) \to B(\infty)$ such that $p_{\lambda} \circ i_{\lambda}$ is the identity on $T_{-\lambda} \otimes B(\lambda)$.

We can also form the \defn{contragrediant dual crystal} $B^{\vee}$ of $B$ as follows. Let $B^{\vee} = \{ b^{\vee} \mid b \in B\}$, and define the crystal structure on $B^{\vee}$ by
\begin{gather*}
f_i(b^{\vee}) = (e_i b)^{\vee},
\hspace{50pt}
e_i(b^{\vee}) = (f_i b)^{\vee}
\\ \varphi_i(b^{\vee}) = \varepsilon_i(b),
\hspace{50pt}
\varepsilon_i(b^{\vee}) = \varphi_i(b),
\\ \wt(b^{\vee}) = -\wt(b),
\end{gather*}
for all $b \in B$.
Note that $(B^{\vee})^{\vee}$ is canonically isomorphic to $B$. We say the $B$ is \defn{dual isomorphic} to $C$ if there exists a crystal isomorphism $\Psi \colon B \to C^{\vee}$ and the canonically induced bijection $\Psi^{\vee} \colon B \to C$ is a \defn{dual crystal isomorphism}. Explicitly, a dual crystal isomorphism satisfies
\begin{gather*}
f_i\bigl( \Psi^{\vee}(b)\bigr) = \Psi^{\vee}(e_i b),
\hspace{50pt}
e_i\bigl( \Psi^{\vee}(b)\bigr) = \Psi^{\vee}(f_i b),
\\ \varphi_i\bigl(\Psi^{\vee}(b)\bigr) = \varepsilon_i(b),
\hspace{50pt}
\varepsilon_i\bigl(\Psi^{\vee}(b)\bigr) = \varphi_i(b),
\\ \wt\bigl(\Psi^{\vee}(b)\bigr) = -\wt(b),
\end{gather*}
for all $b \in B$.

\subsection{Alcove Path Model}
\label{section:alcove_model}
Given $\alpha \in \Phi^+$ and $h \in \ZZ$, let
\begin{equation}
	H_{\alpha, h}:= \left\{  \lambda \in \hh^{*}_{\RR} \mid
	\inner{\lambda}{\alpha^{\vee}} = h \right\}
	\label{eqn:affine_hyperplane}.
\end{equation}
The hyperplanes $H_{\alpha,h}$ divide the real vector space $\hh^{*}_\RR$ into open connected components, called \emph{alcoves.} 
Let $s_{\alpha,h}$ denote the reflection in $\hh^*_{\RR}$ across $H_{\alpha,h}$, and denote $s_{\alpha} := s_{\alpha,0}$.
Let $A_{\circ} = \{ \mu \in \hh^*_{\RR} \mid 0 < \inner{\mu}{\alpha_i^{\vee}} < 1 \text{ for all } i \in I \}$ denote the fundamental alcove, and let $A_{\lambda} = A_{\circ} + \lambda$ be the translation of $A_{\circ}$ by $\lambda$.
Fix some $\lambda \in P^+$. 
For a pair of adjacent alcoves $A$ and $B$, \textit{i.e.}, their closures have non-empty intersection, we write $A \rootarrow{\alpha} B$ if the common wall of $A$ and $B$ is orthogonal to the root $\alpha \in \Phi$ and $\alpha$ points in the direction from $A$ to $B$.
Equate a total ordering on the set
\[
    R_{\lambda} := \{ (\beta, h) \mid \beta\in\Phi^+,\, 0 \le h < \langle \lambda, \beta^\vee \rangle \}
\]
with the sequence $\Gamma = (\beta_{\ind} \mid \beta_{\ind} \in \Phi^+)_{\ind \in \IndSet}$ for some totally ordered countable indexing set $\IndSet = \{\ind_1 < \ind_2 < \cdots < \ind_m\}$ (with possibly $m = \infty$) such that a root $\alpha$ occurs $\inner{\lambda}{\alpha^{\vee}}$ times by $(\beta, h)$ being the $h$-th occurrence of $\beta$ in $\Gamma$.
In other words, to obtain $\Gamma$ from a totally ordered $R_\lambda$, ignore the second index from the tuples $(\beta,k)$.
For each $\ind \in \IndSet$ and $\delta \in \Phi^+$, define $N^{\Gamma}_{\ind}(\delta) := \lvert \{ \ind' < \ind \mid \beta_{\ind'} = \delta \} \rvert$.
A sequence $\Gamma = (\beta_{\ind})_{\ind \in \IndSet}$ is called a \defn{$\lambda$-chain} if it corresponds to a total ordering on $R_{\lambda}$ and for any $\alpha,\beta,\gamma \in \Phi^+$ such that $\alpha \neq \beta$ and $\gamma^{\vee} = \alpha^{\vee} + p \beta^{\vee}$, for some $p \in \ZZ$, then
\begin{equation}
\label{eq:chain_cond}
N^{\Gamma}_{\ind}(\gamma) = N^{\Gamma}_{\ind}(\alpha) + p N^{\Gamma}_{\ind}(\beta)
\end{equation}
for all $\ind \in \IndSet$ such that $\beta = \beta_{\ind}$.
\begin{remark}
    \label{remark:total_order}
    We can recover the total order on $R_{\lambda}$ from $\Gamma$ by $\beta_k \mapsto \bigl( \beta_k, N^{\Gamma}_k( \beta_k) \bigr)$, and we will also refer to this order, viewed as a sequence, as a $ \lambda $-chain.
\end{remark}
If $\lvert R_{\lambda} \rvert = m < \infty$, we can equate a $\lambda$-chain with an alcove path of shortest length from $A_{\circ}$ to $A_{-\lambda}$ by
\[
A_{\circ} = A_0 \rootarrow{-\beta_1} A_1 \rootarrow{-\beta_2} \cdots \rootarrow{-\beta_m} A_m = A_{-\lambda}.
\]

\begin{dfn}
\label{def:concatenation}
Let $\Gamma = (\beta_{\ind})_{\ind \in \IndSet}$ and $\Gamma' = (\beta'_{\ind'})_{\ind' \in \IndSet'}$ be a $\lambda$-chain and a $\lambda'$-chain respectively.
Let $\Gamma \ast \Gamma'$ denote the \defn{concatenated} sequence $(\beta^*_{\ind^*})_{\ind^* \in \IndSet \sqcup \IndSet'}$, where
\[
\beta^*_{\ind^*} = \begin{cases}
\beta_{\ind^*} & \text{if } \ind^* \in \IndSet, \\
\beta'_{\ind^*} & \text{if } \ind^* \in \IndSet',
\end{cases}
\]
and the ordering on $\IndSet \sqcup \IndSet'$ is given by the total orders on $\IndSet$ and $\IndSet'$ and defining $\IndSet < \IndSet'$.
We denote $\Gamma^p$ as $\Gamma$ concatenated with itself $p$ times, and we consider the indexing set to be $\IndSet \times \{1,\dotsc,p\}$.
\end{dfn}

\begin{prop}
The concatenation $\Gamma \ast \Gamma'$ is a $(\lambda+\lambda')$-chain.
\end{prop}

\begin{proof}
It is clear that $\Gamma \ast \Gamma'$ gives a total ordering on $R_{\lambda+\lambda'}$.
Note that Equation~\eqref{eq:chain_cond} is satisfied for $\ind \in \IndSet$ because $N^{\Gamma}_{\ind}(\delta) = N^{\Gamma \ast \Gamma'}_{\ind}(\delta)$ for all $\delta \in \Phi^+$ and $\ind \in \IndSet$ and because $\Gamma$ is a $\lambda$-chain.
Similarly, we have $N^{\Gamma \ast \Gamma'}_{\ind'}(\alpha) = \inner{\lambda}{\alpha^{\vee}} + N^{\Gamma'}_{\ind'}(\alpha)$ for all $\ind' \in \IndSet'$ and $\alpha \in \Phi^+$ since $\IndSet < \ind'$ and $\Gamma$ is a $\lambda$-chain.
Likewise, if for $\gamma^{\vee} = \alpha^{\vee} + p \beta^{\vee}$, then Equation~\eqref{eq:chain_cond} is satisfied for all $\ind' \in \IndSet'$ because
\begin{align*}
N^{\Gamma'}_{\ind'}(\gamma) & = N^{\Gamma'}_{\ind'}(\alpha) + p N^{\Gamma'}_{\ind'}(\beta),
\\ N^{\Gamma'}_{\ind'}(\gamma) & = N^{\Gamma'}_{\ind'}(\alpha) + p N^{\Gamma'}_{\ind'}(\beta) + \inner{\lambda}{\alpha^{\vee} + p\beta^{\vee} - \gamma^{\vee}},
\\ \inner{\lambda}{\gamma^{\vee}} + N^{\Gamma'}_{\ind'}(\gamma) & = \inner{\lambda}{\alpha^{\vee}} + N^{\Gamma'}_{\ind'}(\alpha) + p \bigl( \inner{\lambda}{\beta^{\vee}} + N^{\Gamma'}_{\ind'}(\beta) \bigr),
\\ N^{\Gamma \ast \Gamma'}_{\ind'}(\gamma) & = N^{\Gamma \ast \Gamma'}_{\ind'}(\alpha) + p N^{\Gamma \ast \Gamma'}_{\ind'}(\beta),
\end{align*}
where the first equality holds because $\Gamma'$ is a $\lambda'$-chain.
\end{proof}

\begin{remark}
We note that being a finite $\lambda$-chain means it is a minimal length alcove path to $A_{-\lambda}$~\cite{LP07,LP08}.
Therefore, for two finite chains, their concatenation is a minimal length alcove path to $A_{-\lambda-\mu}$.
\end{remark}

Fix a total order on the set of simple roots $\alpha_1 \prec \alpha_2 \prec \cdots \prec \alpha_n$.
We recall the definition of a particular $\lambda$-chain from~\cite[Prop.~4.2]{LP08} called the \defn{lex $\lambda$-chain} and denoted by $\Gamma_{\lambda}$.
We define the lex total ordering on $R_{\lambda}$ as follows.
For each $(\beta, h) \in R_{\lambda}$, let $\beta^{\vee} = c_1\alpha_1^{\vee} + \cdots + c_r \alpha_n^{\vee}$, and define the vector
\[
     v_{\beta, h} :=\frac{1}{\langle \lambda, \beta^{\vee} \rangle}(h, c_1, \ldots, c_n)
\]
in $\QQ^{n+1}$.
Then define $(\beta, h) < (\beta', h')$ if and only if $v_{\beta, h} < v_{\beta', h'}$ in the lexicographic order on $\QQ^{n+1}$, which defines a total order on $R_{\lambda}$.
That is to say, if the $\ind$-th element of $R_{\lambda}$ with respect to this order is $(\beta, h)$, then set $\beta_{\ind} = \beta$, and $\ell_{\ind} = h$.


Let $r_j = s_{\beta_j}$ and $\widehat{r}_j = s_{\beta_j, -\ell_j}$.
We consider a set of \defn{folding positions} $J = \{j_1 < j_2 < \cdots < j_p\} \subseteq \IndSet$, and call $J$ \defn{admissible} if
we have 
\begin{equation}
\label{eqn:bruhat_path}
\iota(J) := 1 \lessdot r_{j_1} \lessdot r_{j_1} r_{j_2} \lessdot \cdots \lessdot
r_{j_1} r_{j_2} \dotsm r_{j_p} =: \tau(J),
\end{equation}
where $w \lessdot w'$ denotes a cover relation in Bruhat order.
In other words, $J$ is admissible if it corresponds to a path in the Bruhat graph of $W$. 
Let $\Al(\Gamma_{\lambda})$ denote the set of all $J \subseteq \IndSet$ such that $J$ is admissible.
We also write $\Al(\lambda) := \Al(\Gamma_{\lambda})$.
We will identify the integers $j_k$ of an admissible set with the corresponding $j_k$-th element in the $\lambda$-chain; in other words, we identify $\{j_1 < \cdots < j_p\} = \{(\beta_{j_1}, \ell_{j_1}) < \cdots < (\beta_{j_p}, \ell_{j_p})\}$.

Now we recall the crystal structure on $\Al(\lambda)$ from~\cite{LP08}.
First, the weight function $\wt \colon \Al(\lambda) \to P$ is defined as
\begin{equation}
\label{eq:alcove_weight}
\wt(J) = -\widehat{r}_{j_1} \dotsm \widehat{r}_{j_p} (-\lambda).
\end{equation}
Next consider some $J \in \Al(\lambda)$ and define $\Gamma_{\lambda}(J) = (\gamma_{\ind})_{\ind \in \IndSet}$, where
\begin{equation}
    \label{eqn:defn_gamma}
\gamma_{\ind} = r_{j_1} r_{j_2} \cdots r_{j_t} (\beta_{\ind})
\end{equation}
with $t = \max\{a \mid j_a < \ind\}$.
%
%
Next, we describe the crystal operators.
Our description is in  terms of $\Gamma_{\lambda}(J)$, and is equivalent to Equation~\cite{LP08} where it is shown that crystal operators give admissible sequences.
We show this connection in Appendix~\ref{sec:ame}.
Fix some $i \in I$, and we define the sets
\begin{subequations}
\label{eqn:root_op_sets}
\begin{align} 
    I_{\alpha_i} &= \{ \ind \mid \gamma_{\ind} = \pm \alpha_i\}, \label{eqn:root_set_I} \\
    I_{\alpha_i} \setminus J &= \{ d_1 < d_2 < \cdots < d_q \}.
%
\end{align}
\end{subequations}
Consider the word on the alphabet $\{ +, - \}$ given by 
\begin{equation}
\label{eqn:crystal_op_word}
\sgn(\gamma_{d_1})\sgn(\gamma_{d_2}) \dotsm \sgn(\gamma_{d_q}),
\end{equation}
where $\sgn(\gamma)$ is the sign of $\gamma$.
Cancel $- +$ pairs in this word until none remain, and we call this the \defn{reduced $i$-signature}.
If there is no $+$ in the reduced $i$-signature, then define 
\begin{subequations}
    \label{eqn:f_op}
\begin{equation}
    \label{eqn:f_op_a}
f_i J =
\begin{cases}
J \setminus \{\min (J \cap I_{\alpha_i})\}  & \text{if } \inner{ 
    \iota(J)(\rho) 
}
{\alpha_i^{\vee}} < 0,\\
0 & \text{otherwise}.
\end{cases}
\end{equation}
Otherwise, let $a$ be the index corresponding to the rightmost $+$ in the reduced $i$-signature. Let $A = \{ j \in J \cap I_{\alpha_i} \mid j > a\}$, and define
\begin{equation}
f_i J = \begin{cases}
    J \cup \{a\} & \text{if } A = \emptyset, \\
    (J \setminus \{\min A\}) \cup \{a\} & \text{otherwise.}
\end{cases}
\end{equation}
\label{eqn:f_op_b}
\end{subequations}
\begin{remark}
\label{remark:simplify_dual_operators}
Since $\iota(J) = 1$ in Equation~\eqref{eqn:f_op_a}, we have $\inner{\iota(J)(\rho)}{\alpha_i^{\vee}} > 0$, and hence, $f_i J$ will always be $0$ in this case.
The reason for defining $f_i$ this way is to simplify  construction of crystal operators in the dual model in Section~\ref{sec:alcove_reversed}. 
\end{remark}

The definition for $e_i$ is similar.  
If no $-$ exists in the reduced $i$-signature, then define
\begin{subequations}
    \label{eqn:e_op}
\begin{equation}
    \label{eqn:e_op_a}
e_i(J)=
\begin{cases}
J \setminus \{\max (J \cap I_{\alpha_i})\}  & \text{if }
\inner{
    \iota(J) \tau(J)(\rho)
}{\alpha_i^{\vee}} < 0,\\
0 & \text{otherwise}.
\end{cases}
\end{equation}
Otherwise, let $a$ be the index corresponding to the leftmost $-$ in the reduced $i$-signature.
Let $A = \{ j \in J \cap I_{\alpha_i} \mid j < a\}$, and define
\begin{equation}
    \label{eqn:e_op_b}
e_i J = \begin{cases}
    J \cup \{a\} & \text{if } A = \emptyset, \\
    (J \setminus \{\max A \}) \cup \{a\} & \text{otherwise.}
\end{cases}
\end{equation}
\end{subequations}

\begin{remark}
If the reduced $i$-signature contains the symbol $-$, then it can be shown that $A \ne \emptyset$.
The case $A = \emptyset$ is included in Equation~\eqref{eqn:e_op_b} to simplify construction of crystal operators in the dual model in Section~\ref{sec:alcove_reversed}.
\end{remark}

For any $\lambda$-chain $\Gamma$, we define $\varepsilon_i$ and $\varphi_i$ by requiring that $\Al(\Gamma)$ is a regular crystal.

\begin{thm}[{\cite{LP08}}]
\label{thm:isomorphic_highest_weight}
Fix some $\lambda \in P^+$. Then, we have
\[
\Al(\lambda) \iso B(\lambda).
\]
\end{thm}

\newcommand{\alcoveexrho}[2]{
\begin{tikzpicture}[>=stealth,scale=0.7,baseline=-3]
\begin{scope}
\clip (-2.4,-2.75) rectangle (2.4,3.5);
\fill[gray!10] (0,#2) -- +(60:13) -- +(120:13) -- cycle;
\foreach \i in {-7,...,7} {
   \draw (30:\i) +(120:9) -- +(120:-5);
   \draw (-30:\i) +(60:9) -- +(60:-5);
   \draw (90:\i) +(180:9) -- +(180:-5);
}
\draw[thick,->,blue] (0,#2) -- +(60:1.15) node[anchor=west] {$\Lambda_1$};
\draw[thick,->,blue] (0,#2) -- +(120:1.15) node[anchor=east] {$\Lambda_2$};
\draw[thick,->,color=black!40!green] (0,#2) -- +(30:2) node[anchor=north] {$\alpha_1$};
\draw[thick,->,color=black!40!green] (0,#2) -- +(150:2) node[anchor=north] {$\alpha_2$};
\def\len{0.67}   
#1
\fill (0,#2) circle (.1);
\end{scope}
\end{tikzpicture}
} 

\begin{figure}[t]
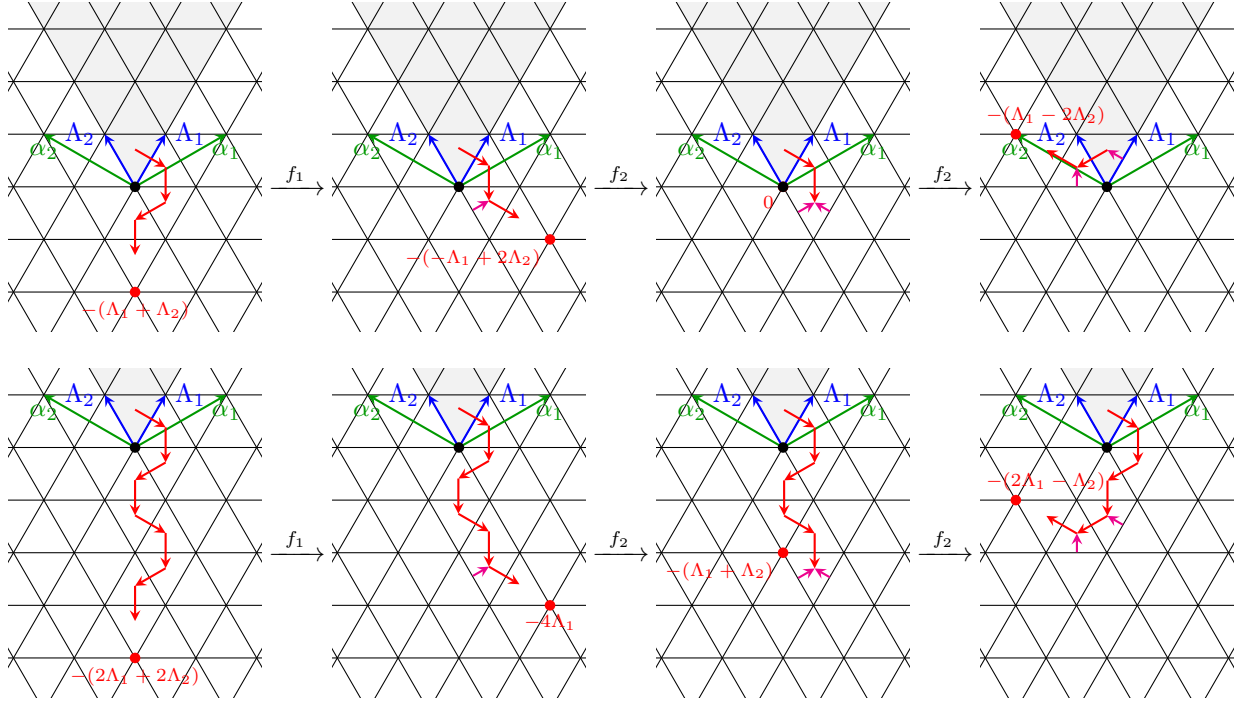

\begin{gather*}
\alcoveexrho{
  \path[thick, red,<-]
  (0,-1.3) edge +(90:\len)
  ++(90:\len) edge +(30:\len)
  ++(30:\len) edge +(90:\len)
  ++(90:\len) edge +(150:\len);
  \fill[red] (0,-2) circle (.1) node[anchor=north] {\scriptsize $-(\Lambda_1 + \Lambda_2)$};
}{0}
\xrightarrow[\hspace{15pt}]{f_1}
\alcoveexrho{
  \path[thick, red,<-]
  (1.15,-0.6) edge +(150:\len)
  ++(150:\len) edge[color=magenta] +(210:\len/2+0.015)
   edge +(90:\len)
  ++(90:\len) edge +(150:\len);
  \fill[red] (1.73,-1) circle (.1) node[anchor=north east] {\scriptsize $-(-\Lambda_1 + 2\Lambda_2)$};
}{0}
\xrightarrow[\hspace{15pt}]{f_2}
\alcoveexrho{
  \path[thick, red,<-]
  (0.6,-0.3) edge[color=magenta] +(-30:\len/2)
  edge[color=magenta] +(210:\len/2+0.025)
   edge +(90:\len)
  ++(90:\len) edge +(150:\len);
  \fill[red] (0,0) circle (.1) node[anchor=north east] {\scriptsize $0$};
}{0}
\xrightarrow[\hspace{15pt}]{f_2}
\alcoveexrho{
  \path[thick, red,<-]
  (-1.15,0.7) edge +(-30:\len)
  ++(-30:\len) edge[color=magenta] +(-90:\len/2+.02)
   edge +(30:\len)
  ++(30:\len) edge[color=magenta] +(-30:\len/2);
  \fill[red] (-1.73,1) circle (.1) node[anchor=south] {\scriptsize \hspace{20pt} $-(\Lambda_1 - 2\Lambda_2)$};
}{0}
\\[1em]
\alcoveexrho{
  \path[thick, red,<-]
  (0,-1.3) edge +(90:\len)
  ++(90:\len) edge +(30:\len)
  ++(30:\len) edge +(90:\len)
  ++(90:\len) edge +(150:\len)
  ++(150:\len) edge +(90:\len) ++(90:\len) edge +(30:\len) ++(30:\len) edge +(90:\len) ++(90:\len) edge +(150:\len);
  \fill[red] (0,-2) circle (.1) node[anchor=north] {\scriptsize $-(2\Lambda_1 + 2\Lambda_2)$};
}{2}
\xrightarrow[\hspace{15pt}]{f_1}
\alcoveexrho{
  \path[thick, red,<-]
  (1.15,-0.6) edge +(150:\len)
  ++(150:\len) edge[color=magenta] +(210:\len/2+0.015)
   edge +(90:\len)
  ++(90:\len) edge +(150:\len)
  ++(150:\len) edge +(90:\len) ++(90:\len) edge +(30:\len) ++(30:\len) edge +(90:\len) ++(90:\len) edge +(150:\len);
  \fill[red] (1.73,-1) circle (.1) node[anchor=north] {\scriptsize $-4\Lambda_1$};
}{2}
\xrightarrow[\hspace{15pt}]{f_2}
\alcoveexrho{
  \path[thick, red,<-]
  (0.6,-0.3) edge[color=magenta] +(-30:\len/2)
  edge[color=magenta] +(210:\len/2+0.025)
   edge +(90:\len)
  ++(90:\len) edge +(150:\len)
  ++(150:\len) edge +(90:\len) ++(90:\len) edge +(30:\len) ++(30:\len) edge +(90:\len) ++(90:\len) edge +(150:\len);
  \fill[red] (0,0) circle (.1) node[anchor=north east] {\scriptsize $-(\Lambda_1 + \Lambda_2)$};
}{2}
\xrightarrow[\hspace{15pt}]{f_2}
\alcoveexrho{
  \path[thick, red,<-]
  (-1.15,0.7) edge +(-30:\len)
  ++(-30:\len) edge[color=magenta] +(-90:\len/2+.02)
   edge +(30:\len)
  ++(30:\len) edge[color=magenta] +(-30:\len/2)
   edge +(90:\len)
  ++(90:\len) edge +(30:\len)
  ++(30:\len) edge +(90:\len)
  ++(90:\len) edge +(150:\len);
  \fill[red] (-1.73,1) circle (.1) node[anchor=south] {\scriptsize \hspace{20pt} $-(2\Lambda_1 - \Lambda_2)$};
}{2}
\end{gather*}
\caption{The action of a few crystal operators on $\Al(\rho)$ (above) and $\Al(2\rho)$ (below) in type $A_2$ starting with $\emptyset$ on the left.}
\label{fig:al_rho_ex}
\end{figure}

%
%

\subsection{Littelmann path model}
\label{sec:littelmann_path}

Let $\pi_1, \pi_2 \colon [0, 1] \to \hh^*_{\RR}$, and define an equivalence relation $\sim$ by saying $\pi_1 \sim \pi_2$ if there exists a piecewise-linear, nondecreasing, surjective, continuous function $\phi \colon [0, 1] \to [0, 1]$ such that $\pi_1 = \pi_2 \circ \phi$. A \defn{path} is an equivalence class $[\pi]$ such that $\pi(0) = 0$. For ease of notation, we will simply write a path by $\pi$.

Let $\pi_1$ and $\pi_2$ be paths.
Define the concatenation $\pi = \pi_1 \ast \pi_2$ by
\[
\pi(t) := \begin{cases} \pi_1(2t) & 0 \leq t \leq 1/2, \\ \pi_1(1) + \pi_2(2t-1) & 1/2 < t \leq 1. \end{cases}
\]
Next, consider a path $\pi$. Define $s_i \pi$ as the path given by $(s_i \pi)(t) = s_i\bigl(\pi(t)\bigr)$.

We now recall the crystal structure on the set of all paths from~\cite{L95,L95-2}.
Fix some $i \in I$ and path $\pi$. Define functions $H_{i,\pi} \colon [0, 1] \to \RR$ by
\[
\pi(t) = \sum_{i \in I} H_{i,\pi}(t) \Lambda_i,
\]
and so $H_{i,\pi}(t) = \inner{\pi(t)}{\alpha_i^{\vee}}$. Let $m_{i,\pi} := \min \{ H_{i,\pi}(t) \mid t \in [0,1] \}$ denote the minimal value of $H_{i,\pi}$.

If $-m_{i,\pi} < 1$, then define $e_i \pi = 0$, otherwise define $e_i \pi$ as the path given by
\[
(e_i \pi)(t) = \begin{cases}
\pi(t) & \text{if } t \leq t_0, \\
\pi(t_0) + s_i\bigl(\pi(t) - \pi(t_0)\bigr) & \text{if } t_0 < t \leq t_1, \\
\pi(t) + \alpha_i & \text{if } t_1 \leq t,
\end{cases}
\]
where
\begin{align*}
t_1 & := \min \{ t \in [0, 1] \mid H_{i,\pi}(t) = m_{i,\pi} \},
\\ t_0 & := \max \{ t \in [0, t_1] \mid H_{i,\pi}(t') \geq m_{i,\pi} + 1 \text{ for all } t' \in [0, t] \}.
\end{align*}
Next, if $H_{i,\pi}(1) - m_{i,\pi} < 1$, then define $f_i \pi = 0$, otherwise define $f_i \pi$ as the path given by
\[
(f_i \pi)(t) = \begin{cases}
\pi(t) & \text{if } t \leq \overline{t}_0, \\
\pi(\overline{t}_0) + s_i\bigl(\pi(t) - \pi(\overline{t}_0)\bigr) & \text{if } \overline{t}_0 < t \leq \overline{t}_1, \\
\pi(t) - \alpha_i & \text{if } \overline{t}_1 \leq t,
\end{cases}
\]
where
\begin{align*}
\overline{t}_0 & := \max \{ t \in [0, 1] \mid H_{i,\pi}(t) = m_{i, \pi} \},
\\ \overline{t}_1 & := \min \{ t \in [\overline{t}_0, 1] \mid H_{i,\pi}(t') \geq m_{i,\pi} + 1 \text{ for } t' \in [t, 1] \}.
\end{align*}

For the remaining crystal structure, we define
\begin{align*}
\varepsilon_i(\pi) & = -m_{i, \pi},
\\ \varphi_i(\pi) & = H_{i,\pi}(1) - m_{i, \pi},
\\ \wt(\pi) & = \pi(1).
\end{align*}

Let $\Pi(\lambda)$ denote the closure under the crystal operators of the path $\pi_{\lambda}(t) = t \lambda$.

\begin{thm}[{\cite{K96,L95,L95-2}}]
Let $\g$ be of symmetrizable type and $\lambda \in P^+$. Then
\[
\Pi(\lambda) \iso B(\lambda).
\]
Furthermore, $\Pi(\lambda)$ is the set of Lakshmibai--Seshadri (LS) paths of shape $\lambda$.
Moreover, $\Pi(\lambda) \otimes \Pi(\mu)$ is isomorphic to $\{ \xi \ast \pi \mid \pi \in \Pi(\lambda), \xi \in \Pi(\mu)\}$ by $\pi \otimes \xi \mapsto \xi \ast \pi$.
\end{thm}

\begin{remark}
The reversal of the concatenation is due to our order of the tensor product. See Remark~\ref{rem:tensor_order}.
\end{remark}

Furthermore, we note that the contragredient dual path $\pi^{\vee}$ is given explicitly by
\begin{equation}
    \label{eqn:dual LS path}
\pi^{\vee}(t) = \pi(1 - t) - \pi(1).
\end{equation}
Moreover, we have $(f_i \pi)^{\vee} = e_i(\pi^{\vee})$.
This gives the following proposition.

\begin{prop}
\label{prop:dual_littelmann_path}
We have $\Pi(-\lambda) \iso \Pi(\lambda)^{\vee}$ given by $\pi \mapsto \pi^{\vee}$.
\end{prop}
For $\lambda \in P^+$ and $\g$ of finite type, the lowest weight element of $\Pi(\lambda)$ is precisely $\pi_{w_0 \lambda}$, where $w_0$ is the longest element of $W$.
Hence, we have $\Pi(-w_0\lambda) = \Pi(\lambda)^{\vee}$ as sets~\cite{L95,L95-2}.

Now we recall the construction of $B(\infty)$ using the (modified) Littelmann paths from~\cite{LZ11}.
An \defn{extended path} is an equivalence class $\pi \colon [0, \infty) \to \hh^*_{\RR}$, with the same equivalence relation $\sim$ above, that eventually results in the direction $\rho$; that is, there exists a $T$ such that for all $t > T$, we have $\pi'(t) = \rho$, where $\pi' = \frac{d\pi}{dt}$.

Define $\Pi(\infty)$ as the closure under the crystal operators of $\pi_{\infty}(t) = t \rho$. For $\Pi(\infty)$, we need to modify the definition of weight and $\varphi_i$ to be
\begin{align*}
\wt(\pi) & = \pi(T) - T\rho,
\\ \varphi_i(\pi) & = \varepsilon_i(\pi) + \inner{\wt(\pi)}{\alpha_i^{\vee}} = -m_{i,\pi} + H_{i,\pi}(T) - T,
\end{align*}
where $T = \min \{ t \mid \pi'(\widetilde{t}) = \rho, \widetilde{t} \geq t \}$, whereas $\varepsilon_i(\pi) = -m_{i, \pi}$ as for $\Pi(\lambda)$.
For the definition of the crystal operators, we replace the intervals $[0, 1]$ with $[0, \infty)$ and drop the condition for $f_i \pi = 0$ (alternatively, it is never satisfied because $\lim_{t \to \infty} H_{i,\pi}(t) - m_{i,\pi} = \infty$).

\begin{thm}[{\cite{LZ11}}]
\label{thm:infinty_LS_paths}
Let $\g$ be of symmetrizable type. Then
\[
\Pi(\infty) \iso B(\infty).
\]
\end{thm}

\subsection{Continuous limit}
\label{sec:continuous_limit}


We recall the dual crystal isomorphism $\varpi_{\lambda} \colon \Al(\lambda) \to \Pi(-\lambda)$ from~\cite[Thm.~9.4]{LP08}.

Consider an admissible set $J = \{(\zeta_1, \ell_1) < \cdots < (\zeta_p, \ell_p)\} \in \Al(\lambda)$.
Let $R_j = s_{\zeta_j}$ and
let $t_j = \ell_j / \inner{\lambda}{\zeta_j^{\vee}}$, and note that $t_1 \leq t_2 \leq \cdots \leq t_p$.
Next define the set
\[
\{ 0 = a_0 < a_1 < a_2 < \cdots < a_q \} := \{0\} \cup \{t_1, \dotsc, t_p\},
\]
which may be of smaller size due to repetition.
For $0 \leq d \leq q$, define $\mu_d := -R_1 \dotsm R_{n_d}(\lambda)$, where $n_d = \max\{ 1 \leq i \leq p \mid a_d = t_i \}$ and we consider $\mu_0 = -\lambda$ if there is no $i$ such that $t_i = a_d$.
Now, we define $\varpi_{\lambda}(J)$ as the Littelmann path $\pi \colon [0, 1] \to \hh_{\RR}$ given by
\begin{equation}
\label{eq:explicit_littelmann_path}
\pi(t) = (t - a_d) \mu_d + \sum_{m=0}^{d-1} (a_{m+1} - a_m) \mu_m,
\end{equation}
for $a_d \leq t \leq a_{d+1}$ and all $0 \leq d \leq q$ with $a_{q+1} = 1$.

\begin{thm}[{\cite{LP08}}]
\label{thm:alcove_to_path}
Let $\g$ be of symmetrizable type. The map $\varpi_{\lambda} \colon \Al(\lambda) \to \Pi(-\lambda)$ is a dual crystal isomorphism.
\end{thm}

Indeed, the map $\varpi_{\lambda}$ is dual in the sense that the map $\varpi^{\vee}_{\lambda} \colon \Al(\lambda) \to \Pi(-\lambda)^{\vee}$ given by $\varpi^{\vee}_{\lambda}(J) := \varpi_{\lambda}(J)^{\vee}$ is a crystal isomorphism.
From Proposition~\ref{prop:dual_littelmann_path}, we can consider $\varpi_{\lambda}^{\vee}$ as a crystal isomorphism $\Al(\lambda) \iso \Pi(\lambda)$.

We can also roughly describe the map $\varpi_{\lambda}$ geometrically as follows.
Define $F$ to be the set of alcoves that contain the origin, and we note that we can tile by $Q$ translates of $F$ (\textit{i.e.}, $F$ is a fundamental domain with respect to translation by elements in $Q$).
For example, in type $A_2$, these are the 6 chambers that form a hexagon and are in bijection with elements of the Weyl group $S_3$.
We then construct the LS path as a slight perturbation of the path corresponding to a folded alcove path and contracting each translate of $F$ to its corresponding element in $Q$.

\subsection{Contragredient dual alcove paths}
\label{sec:alcove_reversed}

We recall an equivalent formulation of the alcove path model from~\cite{Lenart12}.
\begin{dfn}
A sequence $\Gamma^{\vee} = (\beta_{\ind^{\vee}})_{\ind^{\vee} \in \IndSet^{\vee}}$ with $\IndSet^{\vee} = \{\ind^{\vee}_m < \cdots < \ind^{\vee}_2 < \ind^{\vee}_1 \}$ (with possibly $m = \infty$) is a \defn{dual $\lambda$-chain} if it corresponds to a total ordering on $R_{\lambda}$\footnote{Note that the bijection is given by now reading right-to-left.} and for any $\alpha,\beta,\gamma \in \Phi^-$ such that $\alpha \neq \beta$ and $\gamma^{\vee} = \alpha^{\vee} + p \beta^{\vee}$, for some $p \in \ZZ$, then
\begin{equation}
\label{eq:chain_cond_dual}
\check{N}^{\Gamma^{\vee}}_{\ind^{\vee}}(\gamma) = \check{N}^{\Gamma^{\vee}}_{\ind^{\vee}}(\alpha) + p \check{N}^{\Gamma^{\vee}}_{\ind^{\vee}}(\beta),
\end{equation}
where $\check{N}^{\Gamma^{\vee}}_{\ind}(\delta) := \lvert \{ \ind' \geq \ind \mid \beta_{\ind'} = \delta \} \rvert$,
for all $\ind^{\vee} \in \IndSet^{\vee}$ such that $\beta = \beta_{\ind}$.
\end{dfn}
For a $\lambda$-chain $\Gamma = (\beta_{\ind})_{\ind \in \IndSet}$, the corresponding dual $\lambda$-chain is given by $\Gamma^{\vee} = (-\beta_{\ind})_{\ind \in \IndSet^{\vee}}$, where $\IndSet^{\vee}$ is $\IndSet$ in the reverse order, which is also the ``$(-\lambda)$-chain.''
In terms of alcove walks, for the path
\[
A_{\circ} = A_0 \rootarrow{-\beta_1} A_1 \rootarrow{-\beta_2} \cdots \rootarrow{-\beta_m} A_m = A_{-\lambda},
\]
the dual path is given by
\[
A_{\circ} = A'_0 \rootarrow{\beta_m} A'_1 \rootarrow{\beta_{m-1}} \cdots \rootarrow{\beta_1} A'_m = A_{\lambda},
\]
where $A'_i = A_{m-i} + \lambda$.

We reindex the dual $\lambda$-chain by the natural isomorphism $\IndSet \leftrightarrow \IndSet^{\vee}$ so that we can write $\Gamma^{\vee} = (\beta_{\ind})_{\ind \in \IndSet}$ to simplify our notation.
A subset $J = \{ j_1 < j_{2} < \cdots < j_p \} \subset \{1,2, \ldots, m\} $ is \defn{dual admissible} if there exists some $w\in W$ with
\[
w \gtrdot wr_{j_1} \gtrdot wr_{j_1} r_{j_{2}} \gtrdot \ldots \gtrdot
wr_{j_1}r_{j_{2}} \cdots r_{j_p} = 1,
\]
cf. Equation~\eqref{eqn:bruhat_path}.
We set 
\[
     \tau(J) = r_{j_1}r_{j_2} \dotsm r_{j_p}\,,
     \qquad
     \iota(J) = w = \tau(J)^{-1}  \,.
\]
As before we have $\Gamma^{\vee}(J) = (\gamma_1, \gamma_{2}, \dotsc, \gamma_q)$, where
\[
\gamma_{\ind} = wr_{j_1} r_{j_{2}} \cdots r_{j_t} (\beta_{\ind}) = r_{j_{p}} \cdots r_{j_{t+1}}(\beta_{\ind})
\]
with $t = \max\{a \mid j_a \leq \ind \}$.

\begin{remark}
The sequence $\Gamma^{\vee}(J)$ in this section can be obtained by reversing the sequence $\Gamma_{\lambda}(J)$ from Section~\ref{section:alcove_model}.
This construction is analogous to taking the contragredient dual of the Littelmann path $\pi^{\vee}$, cf. Equation~\eqref{eqn:dual LS path}.
\end{remark}

Let $\widetilde{\ell}_i = \inner{\lambda}{\beta_i^{\vee}} - \ell_i$, and 
let $\widehat{r}'_i = s_{\beta_i, \widetilde{\ell}_i}$ then
\begin{equation*}
\wt(J) = \iota(J)\widehat{r}'_{j_1} \dotsm \widehat{r}'_{j_p} (\lambda).
\end{equation*}

Let $\Al^{\vee}(\Gamma)$ be defined as the set of subsets $J \subset \{1, \dotsc, m\}$ which are dual admissible with respect to $\Gamma^{\vee}$.
For brevity, we denote $\Al^{\vee}(\lambda) := \Al^{\vee}(\Gamma_{\lambda})$.
In this case, we define crystal operators $e_i$ and $f_i$ by Equations~\eqref{eqn:e_op} and~\eqref{eqn:f_op} respectively, using the sets $I_{\alpha_i}$ and $I_{\alpha_i} \setminus J$ as defined in Equation~\eqref{eqn:root_op_sets}, and the word
\[
    \sgn(-\gamma_{d_1})\sgn(-\gamma_{d_2}) \dotsm \sgn(-\gamma_{d_q}),
\] 
instead of Equation~\eqref{eqn:crystal_op_word}.
Thus, we have the following.

\begin{prop}
\label{prop:reversing_alcoves}
We have
\[
\Al^{\vee}(\Gamma) \iso \Al(\Gamma)^{\vee},
\]
where the crystal isomorphism is given by $J \mapsto J$.
\end{prop}

Proposition~\ref{prop:reversing_alcoves} and Theorem~\ref{thm:alcove_to_path} gives us the following.

\begin{cor}
\label{cor:dual_dual_iso}
Let $\g$ be of symmetrizable type. Then there exists a dual crystal isomorphism $\varpi^{\vee}_{\lambda} \colon \Al^{\vee}(\lambda) \iso \Pi(\lambda)$.
\end{cor}



\section{Infinite alcove paths}
\label{sec:alcove_infinity}

In this section, we construct the alcove path model for $B(\infty)$ that naturally arises from using alcove paths and dual alcove paths, which we will denote by $\Al(\infty)$ and $\Al^{\vee}(\infty)$, respectively.
In the sequel, we will show that these are isomorphic to the direct limit of $\Al(\lambda)$ and $\Al^{\vee}(\lambda)$ as $\lambda \to \infty$, respectively.

\subsection{The crystal $\Al(\infty)$}

We first give a combinatorial interpretation for $\Al(\infty)$ and then a geometric one.
Fix some $\rho$-chain $\Gamma = (\beta_k)_{k \in K}$. 
We define the \defn{$\infty$-chain of $\Gamma$} as $\dotsm \ast \Gamma \ast \Gamma$, which in terms of alcove walks is
\[
\cdots \rootarrow{-\beta_{m-1}} A_{-m-1} \rootarrow{-\beta_m} A_{-m} \rootarrow{-\beta_1} \cdots \rootarrow{-\beta_{m-1}} A_{-1} \rootarrow{-\beta_m} A_0 = A_{\circ}.
\]
Then $\Al(\infty)$ is the set of all admissible sequences with respect to the above $\infty$-chain of $\Gamma$.  As before, an admissible sequence is a finite set.
Note that if we write the folding positions as $\{(\zeta_1, \ell_1), \dotsc, (\zeta_p, \ell_p)\}$, then we have $\ell_k < 0$ for all $k$.

Geometrically, we start with $\emptyset$ denoting the infinite alcove walk ending at the dominant alcove $A_{\circ}$ and indefinitely repeating backwards along the $\rho$-chain.
All subsequent elements in $\Al(\infty)$ are foldings of this alcove walk. 
In particular, it will not necessarily end in the dominant alcove.
See Figure~\ref{fig:al_geometry} for an example.

\newcommand{\alcoveex}[1]{
\begin{tikzpicture}[>=stealth,scale=0.7,baseline=10]
\begin{scope}
\clip (-2.4,-3.5) rectangle (2.4,5.5);
\fill[gray!10] (0,-2) -- +(60:13) -- +(120:13) -- cycle;
\foreach \i in {-7,...,7} {
   \draw (30:\i) +(120:9) -- +(120:-5);
   \draw (-30:\i) +(60:9) -- +(60:-5);
   \draw (90:\i) +(180:9) -- +(180:-5);
}
\draw[thick,->,blue] (0,-2) -- +(60:1.15) node[anchor=west] {$\Lambda_1$};
\draw[thick,->,blue] (0,-2) -- +(120:1.15) node[anchor=east] {$\Lambda_2$};
\draw[thick,->,color=black!40!green] (0,-2) -- +(30:2) node[anchor=north] {$\alpha_1$};
\draw[thick,->,color=black!40!green] (0,-2) -- +(150:2) node[anchor=north] {$\alpha_2$};
\def\len{0.67}   
#1
\fill (0,-2) circle (.1);
\end{scope}
\end{tikzpicture}
} 

\begin{figure}[t]
\[
\alcoveex{
  \path[thick, red,<-]
  (0,-1.3) edge +(90:\len)
  ++(90:\len) edge +(30:\len)
  ++(30:\len) edge +(90:\len)
  ++(90:\len) edge +(150:\len)
  ++(150:\len) edge +(90:\len) ++(90:\len) edge +(30:\len) ++(30:\len) edge +(90:\len) ++(90:\len) edge +(150:\len)
  ++(150:\len) edge +(90:\len) ++(90:\len) edge +(30:\len) ++(30:\len) edge +(90:\len) ++(90:\len) edge +(150:\len)
  ++(150:\len) edge +(90:\len) ++(90:\len) edge +(30:\len) ++(30:\len) edge +(90:\len) ++(90:\len) edge +(150:\len)
  ++(150:\len) edge +(90:\len) ++(90:\len) edge +(30:\len) ++(30:\len) edge +(90:\len) ++(90:\len) edge +(150:\len);
}
\xrightarrow[\hspace{15pt}]{f_1}
\alcoveex{
  \path[thick, red,<-]
  (1.15,-0.6) edge +(150:\len)
  ++(150:\len) edge[color=magenta] +(210:\len/2+0.015)
   edge +(90:\len)
  ++(90:\len) edge +(150:\len)
  ++(150:\len) edge +(90:\len) ++(90:\len) edge +(30:\len) ++(30:\len) edge +(90:\len) ++(90:\len) edge +(150:\len)
  ++(150:\len) edge +(90:\len) ++(90:\len) edge +(30:\len) ++(30:\len) edge +(90:\len) ++(90:\len) edge +(150:\len)
  ++(150:\len) edge +(90:\len) ++(90:\len) edge +(30:\len) ++(30:\len) edge +(90:\len) ++(90:\len) edge +(150:\len)
  ++(150:\len) edge +(90:\len) ++(90:\len) edge +(30:\len) ++(30:\len) edge +(90:\len) ++(90:\len) edge +(150:\len);
}
\xrightarrow[\hspace{15pt}]{f_2}
\alcoveex{
  \path[thick, red,<-]
  (0.6,-0.3) edge[color=magenta] +(-30:\len/2)
  edge[color=magenta] +(210:\len/2+0.025)
   edge +(90:\len)
  ++(90:\len) edge +(150:\len)
  ++(150:\len) edge +(90:\len) ++(90:\len) edge +(30:\len) ++(30:\len) edge +(90:\len) ++(90:\len) edge +(150:\len)
  ++(150:\len) edge +(90:\len) ++(90:\len) edge +(30:\len) ++(30:\len) edge +(90:\len) ++(90:\len) edge +(150:\len)
  ++(150:\len) edge +(90:\len) ++(90:\len) edge +(30:\len) ++(30:\len) edge +(90:\len) ++(90:\len) edge +(150:\len)
  ++(150:\len) edge +(90:\len) ++(90:\len) edge +(30:\len) ++(30:\len) edge +(90:\len) ++(90:\len) edge +(150:\len);
}
\xrightarrow[\hspace{15pt}]{f_2}
\alcoveex{
  \path[thick, red,<-]
  (-1.15,0.7) edge +(-30:\len)
  ++(-30:\len) edge[color=magenta] +(-90:\len/2+.02)
   edge +(30:\len)
  ++(30:\len) edge[color=magenta] +(-30:\len/2)
   edge +(90:\len)
  ++(90:\len) edge +(30:\len)
  ++(30:\len) edge +(90:\len)
  ++(90:\len) edge +(150:\len)
  ++(150:\len) edge +(90:\len) ++(90:\len) edge +(30:\len) ++(30:\len) edge +(90:\len) ++(90:\len) edge +(150:\len)
  ++(150:\len) edge +(90:\len) ++(90:\len) edge +(30:\len) ++(30:\len) edge +(90:\len) ++(90:\len) edge +(150:\len)
  ++(150:\len) edge +(90:\len) ++(90:\len) edge +(30:\len) ++(30:\len) edge +(90:\len) ++(90:\len) edge +(150:\len);
}
\]
\caption{The action of a few crystal operators on $\Al(\infty)$ in type $A_2$ starting with $\emptyset$ on the left.}
\label{fig:al_geometry}
\end{figure}

We define $f_i$ and $e_i$ by Equations~\eqref{eqn:f_op} and~\eqref{eqn:e_op}, respectively, $\varepsilon_i$ by specifying $\Al(\infty)$ is an upper regular crystal, and $\wt$ by Equation~\eqref{eq:alcove_weight} with $\lambda = 0$.
Thus, we can define $\varphi_i$ by Condition~(1) of an abstract $U_q(\g)$-crystal.

\begin{lemma}
The set $\Al(\infty)$ is an abstract $U_q(\g)$-crystal with the crystal structure given above.
\end{lemma}

\begin{proof}
First note that any reduced $i$-signature is of the form $\cdots + + - - \cdots -$, where there are at most $\varepsilon_i(J)$ number of $-$'s.
Thus, the crystal operators $e_i$ and $f_i$ are well-defined.
Next, note that $\emptyset$ is the highest weight element of $\Al(\infty)$, and we have
\[
\varepsilon_i(\emptyset) = \varphi_i(\emptyset) = \inner{\alpha_i^{\vee}}{\wt(\emptyset)} = 0
\]
for all $i \in I$. Thus it is sufficient to show that Conditions~(3) and~(4) hold as it is clear $\varphi_i(J) > -\infty$ for all $J \in \Al(\infty)$. However, these follow from similar arguments as given in~\cite[Sec.~7]{LP08}.
\end{proof}

\begin{figure}[t]
\[
\begin{tikzpicture}[>=latex,line join=bevel,xscale=0.43,yscale=0.7,every node/.style={scale=0.45}]
\node (node_20) at (789.0bp,79.5bp) [draw,draw=none] {$\left(\left(\alpha_{1}, -2\right), \left(\alpha_{1} + \alpha_{2}, -1\right)\right)$};
  \node (node_21) at (922.0bp,8.5bp) [draw,draw=none] {$\left(\left(\alpha_{1}, -3\right), \left(\alpha_{1} + \alpha_{2}, -1\right)\right)$};
  \node (node_9) at (524.0bp,8.5bp) [draw,draw=none] {$\left(\left(\alpha_{2}, -1\right), \left(\alpha_{1} + \alpha_{2}, -2\right), \left(\alpha_{1}, -1\right)\right)$};
  \node (node_8) at (922.0bp,79.5bp) [draw,draw=none] {$\left(\left(\alpha_{1}, -3\right)\right)$};
  \node (node_7) at (387.0bp,79.5bp) [draw,draw=none] {$\left(\left(\alpha_{2}, -1\right), \left(\alpha_{1} + \alpha_{2}, -2\right)\right)$};
  \node (node_6) at (667.0bp,8.5bp) [draw,draw=none] {$\left(\left(\alpha_{2}, -2\right), \left(\alpha_{1}, -1\right)\right)$};
  \node (node_5) at (682.0bp,150.5bp) [draw,draw=none] {$\left(\left(\alpha_{1}, -1\right), \left(\alpha_{1} + \alpha_{2}, -1\right)\right)$};
  \node (node_4) at (363.0bp,150.5bp) [draw,draw=none] {$\left(\left(\alpha_{2}, -1\right), \left(\alpha_{1} + \alpha_{2}, -1\right)\right)$};
  \node (node_3) at (682.0bp,221.5bp) [draw,draw=none] {$\left(\left(\alpha_{1}, -1\right)\right)$};
  \node (node_2) at (126.0bp,8.5bp) [draw,draw=none] {$\left(\left(\alpha_{2}, -3\right), \left(\alpha_{1} + \alpha_{2}, -1\right)\right)$};
  \node (node_1) at (789.0bp,8.5bp) [draw,draw=none] {$\left(\left(\alpha_{1}, -2\right), \left(\alpha_{1} + \alpha_{2}, -2\right)\right)$};
  \node (node_0) at (25.0bp,8.5bp) [draw,draw=none] {$\left(\left(\alpha_{2}, -4\right)\right)$};
  \node (node_19) at (259.0bp,8.5bp) [draw,draw=none] {$\left(\left(\alpha_{2}, -2\right), \left(\alpha_{1} + \alpha_{2}, -2\right)\right)$};
  \node (node_18) at (557.0bp,292.5bp) [draw,draw=none] {$\left(\right)$};
  \node (node_17) at (381.0bp,8.5bp) [draw,draw=none] {$\left(\left(\alpha_{1}, -2\right), \left(\alpha_{2}, -1\right)\right)$};
  \node (node_16) at (254.0bp,150.5bp) [draw,draw=none] {$\left(\left(\alpha_{2}, -2\right)\right)$};
  \node (node_15) at (667.0bp,79.5bp) [draw,draw=none] {$\left(\left(\alpha_{2}, -1\right), \left(\alpha_{1}, -1\right)\right)$};
  \node (node_14) at (1023.0bp,8.5bp) [draw,draw=none] {$\left(\left(\alpha_{1}, -4\right)\right)$};
  \node (node_13) at (254.0bp,79.5bp) [draw,draw=none] {$\left(\left(\alpha_{2}, -2\right), \left(\alpha_{1} + \alpha_{2}, -1\right)\right)$};
  \node (node_12) at (789.0bp,150.5bp) [draw,draw=none] {$\left(\left(\alpha_{1}, -2\right)\right)$};
  \node (node_11) at (126.0bp,79.5bp) [draw,draw=none] {$\left(\left(\alpha_{2}, -3\right)\right)$};
  \node (node_10) at (363.0bp,221.5bp) [draw,draw=none] {$\left(\left(\alpha_{2}, -1\right)\right)$};
  \draw [blue,->] (node_20) ..controls (826.37bp,59.112bp) and (871.11bp,35.903bp)  .. (node_21);
  \definecolor{strokecol}{rgb}{0.0,0.0,0.0};
  \pgfsetstrokecolor{strokecol}
  \draw (879.5bp,44.0bp) node {$1$};
  \draw [red,->] (node_5) ..controls (678.07bp,131.44bp) and (673.96bp,112.5bp)  .. (node_15);
  \draw (684.5bp,115.0bp) node {$2$};
  \draw [blue,->] (node_7) ..controls (385.43bp,60.442bp) and (383.78bp,41.496bp)  .. (node_17);
  \draw (393.5bp,44.0bp) node {$1$};
  \draw [blue,->] (node_18) ..controls (582.08bp,277.65bp) and (630.23bp,251.08bp)  .. (node_3);
  \draw (643.5bp,257.0bp) node {$1$};
  \draw [blue,->] (node_15) ..controls (626.71bp,59.06bp) and (578.3bp,35.7bp)  .. (node_9);
  \draw (621.5bp,44.0bp) node {$1$};
  \draw [blue,->] (node_13) ..controls (255.31bp,60.442bp) and (256.68bp,41.496bp)  .. (node_19);
  \draw (266.5bp,44.0bp) node {$1$};
  \draw [blue,->] (node_11) ..controls (126.0bp,60.442bp) and (126.0bp,41.496bp)  .. (node_2);
  \draw (134.5bp,44.0bp) node {$1$};
  \draw [blue,->] (node_3) ..controls (711.86bp,201.25bp) and (746.3bp,179.03bp)  .. (node_12);
  \draw (756.5bp,186.0bp) node {$1$};
  \draw [red,->] (node_10) ..controls (332.5bp,201.2bp) and (297.18bp,178.83bp)  .. (node_16);
  \draw (330.5bp,186.0bp) node {$2$};
  \draw [red,->] (node_7) ..controls (425.5bp,59.112bp) and (471.58bp,35.903bp)  .. (node_9);
  \draw (480.5bp,44.0bp) node {$2$};
  \draw [red,->] (node_16) ..controls (217.9bp,130.04bp) and (175.6bp,107.24bp)  .. (node_11);
  \draw (214.5bp,115.0bp) node {$2$};
  \draw [red,->] (node_13) ..controls (217.9bp,59.04bp) and (175.6bp,36.236bp)  .. (node_2);
  \draw (214.5bp,44.0bp) node {$2$};
  \draw [blue,->] (node_10) ..controls (363.0bp,202.44bp) and (363.0bp,183.5bp)  .. (node_4);
  \draw (371.5bp,186.0bp) node {$1$};
  \draw [blue,->] (node_12) ..controls (826.37bp,130.11bp) and (871.11bp,106.9bp)  .. (node_8);
  \draw (879.5bp,115.0bp) node {$1$};
  \draw [red,->] (node_3) ..controls (682.0bp,202.44bp) and (682.0bp,183.5bp)  .. (node_5);
  \draw (690.5bp,186.0bp) node {$2$};
  \draw [red,->] (node_15) ..controls (667.0bp,60.442bp) and (667.0bp,41.496bp)  .. (node_6);
  \draw (675.5bp,44.0bp) node {$2$};
  \draw [red,->] (node_12) ..controls (789.0bp,131.44bp) and (789.0bp,112.5bp)  .. (node_20);
  \draw (797.5bp,115.0bp) node {$2$};
  \draw [red,->] (node_18) ..controls (524.29bp,279.87bp) and (439.75bp,249.8bp)  .. (node_10);
  \draw (492.5bp,257.0bp) node {$2$};
  \draw [blue,->] (node_4) ..controls (369.32bp,131.34bp) and (376.0bp,112.12bp)  .. (node_7);
  \draw (386.5bp,115.0bp) node {$1$};
  \draw [red,->] (node_20) ..controls (789.0bp,60.442bp) and (789.0bp,41.496bp)  .. (node_1);
  \draw (797.5bp,44.0bp) node {$2$};
  \draw [red,->] (node_11) ..controls (97.894bp,59.299bp) and (65.591bp,37.231bp)  .. (node_0);
  \draw (96.5bp,44.0bp) node {$2$};
  \draw [blue,->] (node_8) ..controls (950.11bp,59.299bp) and (982.41bp,37.231bp)  .. (node_14);
  \draw (992.5bp,44.0bp) node {$1$};
  \draw [blue,->] (node_5) ..controls (711.86bp,130.25bp) and (746.3bp,108.03bp)  .. (node_20);
  \draw (756.5bp,115.0bp) node {$1$};
  \draw [red,->] (node_4) ..controls (332.5bp,130.2bp) and (297.18bp,107.83bp)  .. (node_13);
  \draw (330.5bp,115.0bp) node {$2$};
  \draw [red,->] (node_8) ..controls (922.0bp,60.442bp) and (922.0bp,41.496bp)  .. (node_21);
  \draw (930.5bp,44.0bp) node {$2$};
  \draw [blue,->] (node_16) ..controls (254.0bp,131.44bp) and (254.0bp,112.5bp)  .. (node_13);
  \draw (262.5bp,115.0bp) node {$1$};
\end{tikzpicture}
\]
\caption{The first four levels of $\Al(\infty)$ of type $A_2$.}
\label{fig:A2_example}
\end{figure}
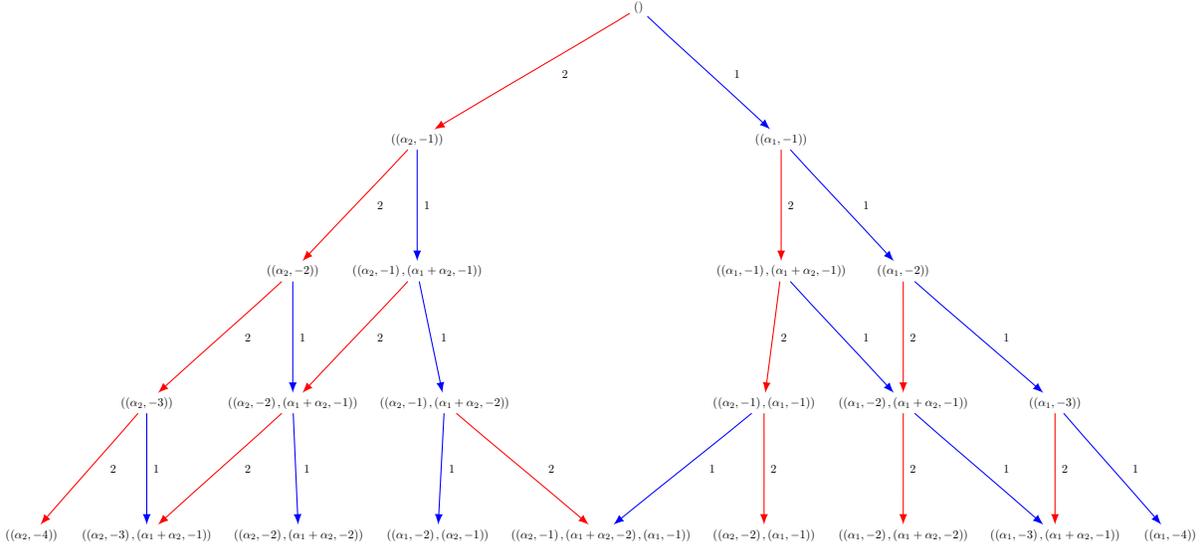

\begin{ex}
We give $f_{i_4} f_{i_3} f_{i_2} f_{i_1} \emptyset$, where $i_1, i_2, i_3, i_4 \in \{1,2\}$, for type $A_2$ in Figure~\ref{fig:A2_example}.
\end{ex}

\subsection{The dual crystal $\Al^{\vee}(\infty)$}

The construction of $\Al^{\vee}(\infty)$ will be similar to the construction done in Section~\ref{sec:alcove_reversed}.

For a fixed dual $\rho$-chain $\Gamma^{\vee} = (\beta_{\ind})_{\ind \in \IndSet}$, we define the \defn{dual $\infty$-chain of $\Gamma$} as $\Gamma^{\vee} \ast \Gamma^{\vee} \ast \cdots$, which in terms of alcove walks is
\[
A_{\circ} = A_0 \rootarrow{\beta_1} A_1 \rootarrow{\beta_2} \cdots \rootarrow{\beta_m} A_m \rootarrow{\beta_1} A_{m+1} \rootarrow{\beta_2} \cdots. 
\]
We can also define crystal operators on $\Al^{\vee}(\infty)$ as in Section~\ref{sec:alcove_reversed}.

\begin{prop}
\label{prop:reversing_infinite_alcoves}
We have a crystal isomorphism
\[
\Al^{\vee}(\infty) \iso \Al(\infty)^{\vee}
\]
given by $J \mapsto J$.
\end{prop}

\begin{proof}
Similar to the proof of Proposition~\ref{prop:reversing_alcoves}.
\end{proof}

Unlike for the model $\Al(\infty)$, the alcove paths for $\Al^{\vee}(\infty)$ will always end in an (closed) alcove that contains the origin (\textit{i.e.}, it will start in an alcove of the fundamental domain with respect to the action of $Q$).
For an example, see Figure~\ref{fig:al_dual_geometry}.

\begin{figure}[t]
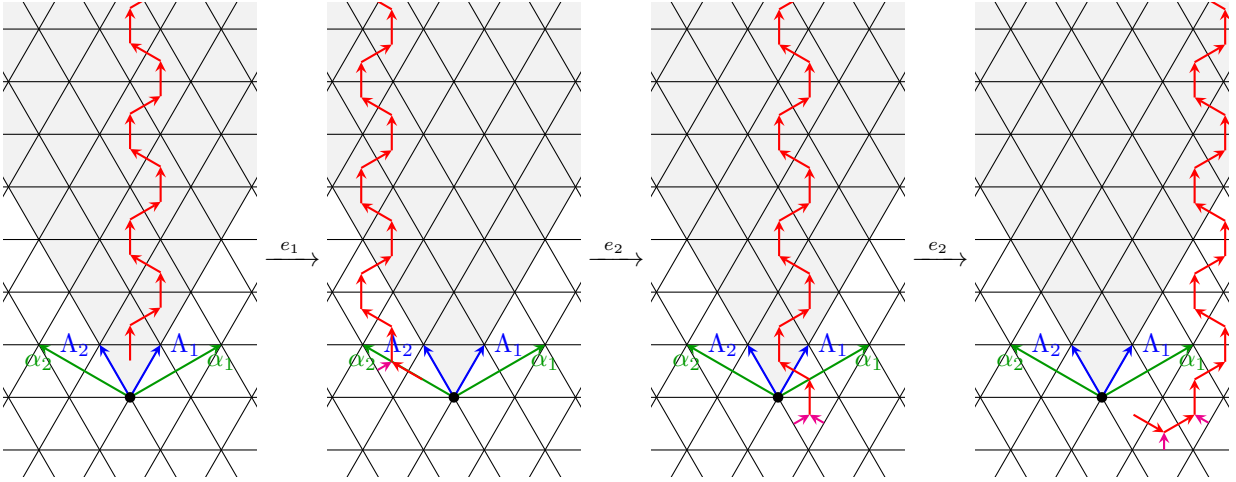

\[
\alcoveex{
  \path[thick, red,->]
  (0,-1.3) edge +(90:\len)
  ++(90:\len) edge +(30:\len)
  ++(30:\len) edge +(90:\len)
  ++(90:\len) edge +(150:\len)
  ++(150:\len) edge +(90:\len) ++(90:\len) edge +(30:\len) ++(30:\len) edge +(90:\len) ++(90:\len) edge +(150:\len)
  ++(150:\len) edge +(90:\len) ++(90:\len) edge +(30:\len) ++(30:\len) edge +(90:\len) ++(90:\len) edge +(150:\len)
  ++(150:\len) edge +(90:\len) ++(90:\len) edge +(30:\len) ++(30:\len) edge +(90:\len) ++(90:\len) edge +(150:\len)
  ++(150:\len) edge +(90:\len) ++(90:\len) edge +(30:\len) ++(30:\len) edge +(90:\len) ++(90:\len) edge +(150:\len);
}
\xrightarrow[\hspace{15pt}]{e_1}
\alcoveex{
  \path[thick, red,->]
  (-0.6,-1.66) edge +(150:\len)
  ++(150:\len) edge[color=magenta,<-] +(210:\len/2-0.025)
   edge +(90:\len)
  ++(90:\len) edge +(150:\len)
  ++(150:\len) edge +(90:\len) ++(90:\len) edge +(30:\len) ++(30:\len) edge +(90:\len) ++(90:\len) edge +(150:\len)
  ++(150:\len) edge +(90:\len) ++(90:\len) edge +(30:\len) ++(30:\len) edge +(90:\len) ++(90:\len) edge +(150:\len)
  ++(150:\len) edge +(90:\len) ++(90:\len) edge +(30:\len) ++(30:\len) edge +(90:\len) ++(90:\len) edge +(150:\len)
  ++(150:\len) edge +(90:\len) ++(90:\len) edge +(30:\len) ++(30:\len) edge +(90:\len) ++(90:\len) edge +(150:\len);
}
\xrightarrow[\hspace{15pt}]{e_2}
\alcoveex{
  \path[thick, red,->]
  (0.6,-2.33) edge[color=magenta,<-] +(-30:\len/2-0.025)
   edge[color=magenta,<-] +(210:\len/2+0.01)
   edge +(90:\len)
  ++(90:\len) edge +(150:\len)
  ++(150:\len) edge +(90:\len) ++(90:\len) edge +(30:\len) ++(30:\len) edge +(90:\len) ++(90:\len) edge +(150:\len)
  ++(150:\len) edge +(90:\len) ++(90:\len) edge +(30:\len) ++(30:\len) edge +(90:\len) ++(90:\len) edge +(150:\len)
  ++(150:\len) edge +(90:\len) ++(90:\len) edge +(30:\len) ++(30:\len) edge +(90:\len) ++(90:\len) edge +(150:\len)
  ++(150:\len) edge +(90:\len) ++(90:\len) edge +(30:\len) ++(30:\len) edge +(90:\len) ++(90:\len) edge +(150:\len);
}
\xrightarrow[\hspace{15pt}]{e_2}
\alcoveex{
  \path[thick, red,->]
  (0.6,-2.33) edge +(-30:\len)
  ++(-30:\len) edge[color=magenta,<-] +(-90:\len/2)
   edge +(30:\len)
  ++(30:\len) edge[color=magenta,<-] +(-30:\len/2-0.03)
   edge +(90:\len)
  ++(90:\len) edge +(30:\len) ++(30:\len) edge +(90:\len) ++(90:\len) edge +(150:\len)
  ++(150:\len) edge +(90:\len) ++(90:\len) edge +(30:\len) ++(30:\len) edge +(90:\len) ++(90:\len) edge +(150:\len)
  ++(150:\len) edge +(90:\len) ++(90:\len) edge +(30:\len) ++(30:\len) edge +(90:\len) ++(90:\len) edge +(150:\len)
  ++(150:\len) edge +(90:\len) ++(90:\len) edge +(30:\len) ++(30:\len) edge +(90:\len) ++(90:\len) edge +(150:\len);
}
\]
\caption{The action of a few crystal operators on $\Al^{\vee}(\infty)$ in type $A_2$ starting with $\emptyset$ on the left.}
\label{fig:al_dual_geometry}
\end{figure}

\section{Main results}
\label{sec:results}

Let $\lambda, \mu \in P^+$.
We embed $\Al(\lambda)$ into $\Al(\mu + \lambda)$ as follows.
Recall that there is a unique component $B(\mu + \lambda) \subseteq B(\mu) \otimes B(\lambda)$ and that $B(\lambda)$ embeds into $B(\mu + \lambda) \subseteq B(\mu) \otimes B(\lambda)$ by $b \mapsto u_{\mu} \otimes b$.
Using this idea and that concatenation corresponds to tensor products in terms of Littelmann paths, we 
consider the $(\mu + \lambda)$-chain $\Delta = \Gamma_{\mu} \ast \Gamma_{\lambda}$.
We define an embedding of $\Al(\lambda)$ into $\Al(\Delta)$ (which is conjecturally isomorphic to $\Al(\lambda + \mu)$) by a $\mu$-shift.
More precisely, let  $\{ (\zeta_1, \ell_1), \dotsc, (\zeta_p, \ell_p) \} = J \in \Al(\lambda)$, and define $S_{\mu} \colon \Al(\lambda) \to T_{-\mu} \otimes \Al(\Delta)$ by
\begin{equation}
\label{eq:shift_map}
S_{\mu}(J) := t_{-\mu} \otimes \bigl( (\zeta_1, \inner{\mu}{\zeta_1^{\vee}} + \ell_1),
\dotsc, (\zeta_p, \inner{\mu}{\zeta_p^{\vee}} + \ell_p) \bigr).
\end{equation}
Observe that $S_{\mu}$ is a crystal embedding since $S_{\mu}(\emptyset) = \emptyset$ and if $f_i J \neq 0$, then $f_i$ either adds a folding position to $J$ or moves a folding position.
This operation depends entirely on the folded $\lambda$-chain and acts on the highest level possible.
Therefore, it is not affected by the shift.
In other words, we have $S_{\mu}(f_i J) = f_i S_{\mu}(J)$.
Similar statements hold for $e_i$.

\begin{remark}
The admissible $\Al(\Delta)$ does not directly correspond to the concatenation of admissible sequences in $\Al(\mu) \otimes \Al(\lambda)$.
See Example~\ref{ex:concat_admissible}.
However, if $\lambda + \mu = k \lambda$ for some $k \in \QQ_{\geq 0}$ (\textit{i.e.}, $\lambda$ and $\mu$ are scalar multiples of some other dominant weight $\nu$), then $\Delta = \Gamma_{k\lambda}$ and $\Al(\Delta) = \Al(k\lambda)$.
\end{remark}

\begin{lemma}
\label{lemma:shift_morphism}
Fix some $\lambda \in P^+$.
Suppose $\mu \in P$ is such that $\mu + \lambda = k \lambda$ for some $k \in \QQ_{\geq 0}$.
The map $S_{\mu} \colon \Al(\lambda) \to T_{-\mu} \otimes \Al(\lambda + \mu)$ given by Equation~\eqref{eq:shift_map}, where $S_{\mu}(J) = 0$ if the result is not admissible, is a crystal morphism.
Moreover, if $k \geq 1$, then $S_{\mu}$ is a crystal embedding, and if $k \leq 1$, then $S_{\mu}$ is a surjection.
\end{lemma}

\begin{proof}
From our assumptions, there exists some $\nu \in P^+$ such that $\Gamma_{\mu+\lambda} = \Gamma_{\nu}^{m_{\mu+\lambda}}$ and $\Gamma_{\lambda} = \Gamma_{\nu}^{m_{\lambda}}$ for some integers $m_{\mu+\lambda}$ and $m_{\lambda}$.
Note that $k = m_{\mu + \lambda} / m_{\lambda}$.
If $m_{\lambda} \leq m_{\mu+\lambda}$, then similar to the discussion above, the crystal operators act only on the $\lambda$-chain part of the $(\mu + \lambda)$-chain.
Furthermore, it is straightforward to see that every admissible sequence in $\Al(\lambda)$ is admissible in $\Al(\lambda + \mu)$.
Likewise, if $m_{\lambda} \geq m_{\mu+\lambda}$, then the crystal operators act only on the $(\mu + \lambda)$-chain part of the $\lambda$-chain and $\Al(\lambda + \mu) \subseteq \Al(\lambda)$.
Thus, the claim follows.
\end{proof}

For an example of Lemma~\ref{lemma:shift_morphism}, compare the top and bottom examples of Figure~\ref{fig:al_rho_ex}.

Next, for $k \geq 0$, define $S^{\inc}_{-k\rho} \colon T_{-k\rho} \otimes \Al(k\rho) \to \Al(\infty)$ by
\[
S^{\inc}_{-k\rho}(t_{-k\rho} \otimes J) = \bigl( (\zeta_1, \ell_1 - k \inner{\rho}{\zeta_1^{\vee}}),
\dotsc, \left(\zeta_p, \ell_p - k \inner{\rho}{\zeta_p^{\vee}} \right) \bigr)
\]
and $S^{\pr}_{k\rho} \colon \Al(\infty) \to T_{-k\rho} \otimes \Al(k\rho)$ by
\[
S^{\pr}_{k\rho}(J) = t_{-k\rho} \otimes \bigl( (\zeta_1, \ell_1 + k \inner{\rho}{\zeta_1^{\vee}}),
\dotsc, \left(\zeta_p, \ell_p + k \inner{\rho}{\zeta_p^{\vee}}\right) \bigr)
\]
if the result is admissible and $S^{\pr}_{k\rho}(J) = 0$ otherwise.

\begin{figure}
\[
\begin{tikzpicture}[>=latex,line join=bevel,xscale=0.43,yscale=0.7,every node/.style={scale=0.45}]
\node (node_20) at (789.0bp,79.5bp) [draw,draw=none] {$\left(\left(\alpha_{1}, 2\right), \left(\alpha_{1} + \alpha_{2}, 7\right)\right)$};
  \node (node_21) at (922.0bp,8.5bp) [draw,draw=none] {$\left(\left(\alpha_{1}, 1\right), \left(\alpha_{1} + \alpha_{2}, 7\right)\right)$};
  \node (node_9) at (524.0bp,8.5bp) [draw,draw=none] {$\left(\left(\alpha_{2}, 3\right), \left(\alpha_{1} + \alpha_{2}, 6\right), \left(\alpha_{1}, 3\right)\right)$};
  \node (node_8) at (922.0bp,79.5bp) [draw,draw=none] {$\left(\left(\alpha_{1}, 1\right)\right)$};
  \node (node_7) at (387.0bp,79.5bp) [draw,draw=none] {$\left(\left(\alpha_{2}, 3\right), \left(\alpha_{1} + \alpha_{2}, 6\right)\right)$};
  \node (node_6) at (667.0bp,8.5bp) [draw,draw=none] {$\left(\left(\alpha_{2}, 2\right), \left(\alpha_{1}, 3\right)\right)$};
  \node (node_5) at (682.0bp,150.5bp) [draw,draw=none] {$\left(\left(\alpha_{1}, 3\right), \left(\alpha_{1} + \alpha_{2}, 7\right)\right)$};
  \node (node_4) at (363.0bp,150.5bp) [draw,draw=none] {$\left(\left(\alpha_{2}, 3\right), \left(\alpha_{1} + \alpha_{2}, 7\right)\right)$};
  \node (node_3) at (682.0bp,221.5bp) [draw,draw=none] {$\left(\left(\alpha_{1}, 3\right)\right)$};
  \node (node_2) at (126.0bp,8.5bp) [draw,draw=none] {$\left(\left(\alpha_{2}, 1\right), \left(\alpha_{1} + \alpha_{2}, 7\right)\right)$};
  \node (node_1) at (789.0bp,8.5bp) [draw,draw=none] {$\left(\left(\alpha_{1}, 2\right), \left(\alpha_{1} + \alpha_{2}, 6\right)\right)$};
  \node (node_0) at (25.0bp,8.5bp) [draw,draw=none] {$\left(\left(\alpha_{2}, 0\right)\right)$};
  \node (node_19) at (259.0bp,8.5bp) [draw,draw=none] {$\left(\left(\alpha_{2}, 2\right), \left(\alpha_{1} + \alpha_{2}, 6\right)\right)$};
  \node (node_18) at (557.0bp,292.5bp) [draw,draw=none] {$\left(\right)$};
  \node (node_17) at (381.0bp,8.5bp) [draw,draw=none] {$\left(\left(\alpha_{1}, 2\right), \left(\alpha_{2}, 3\right)\right)$};
  \node (node_16) at (254.0bp,150.5bp) [draw,draw=none] {$\left(\left(\alpha_{2}, 2\right)\right)$};
  \node (node_15) at (667.0bp,79.5bp) [draw,draw=none] {$\left(\left(\alpha_{2}, 3\right), \left(\alpha_{1}, 3\right)\right)$};
  \node (node_14) at (1023.0bp,8.5bp) [draw,draw=none] {$\left(\left(\alpha_{1}, 0\right)\right)$};
  \node (node_13) at (254.0bp,79.5bp) [draw,draw=none] {$\left(\left(\alpha_{2}, 2\right), \left(\alpha_{1} + \alpha_{2}, 7\right)\right)$};
  \node (node_12) at (789.0bp,150.5bp) [draw,draw=none] {$\left(\left(\alpha_{1}, 4\right)\right)$};
  \node (node_11) at (126.0bp,79.5bp) [draw,draw=none] {$\left(\left(\alpha_{2}, 1\right)\right)$};
  \node (node_10) at (363.0bp,221.5bp) [draw,draw=none] {$\left(\left(\alpha_{2}, 3\right)\right)$};
  \draw [blue,->] (node_20) ..controls (826.37bp,59.112bp) and (871.11bp,35.903bp)  .. (node_21);
  \definecolor{strokecol}{rgb}{0.0,0.0,0.0};
  \pgfsetstrokecolor{strokecol}
  \draw (879.5bp,44.0bp) node {$1$};
  \draw [red,->] (node_5) ..controls (678.07bp,131.44bp) and (673.96bp,112.5bp)  .. (node_15);
  \draw (684.5bp,115.0bp) node {$2$};
  \draw [blue,->] (node_7) ..controls (385.43bp,60.442bp) and (383.78bp,41.496bp)  .. (node_17);
  \draw (393.5bp,44.0bp) node {$1$};
  \draw [blue,->] (node_18) ..controls (582.08bp,277.65bp) and (630.23bp,251.08bp)  .. (node_3);
  \draw (643.5bp,257.0bp) node {$1$};
  \draw [blue,->] (node_15) ..controls (626.71bp,59.06bp) and (578.3bp,35.7bp)  .. (node_9);
  \draw (621.5bp,44.0bp) node {$1$};
  \draw [blue,->] (node_13) ..controls (255.31bp,60.442bp) and (256.68bp,41.496bp)  .. (node_19);
  \draw (266.5bp,44.0bp) node {$1$};
  \draw [blue,->] (node_11) ..controls (126.0bp,60.442bp) and (126.0bp,41.496bp)  .. (node_2);
  \draw (134.5bp,44.0bp) node {$1$};
  \draw [blue,->] (node_3) ..controls (711.86bp,201.25bp) and (746.3bp,179.03bp)  .. (node_12);
  \draw (756.5bp,186.0bp) node {$1$};
  \draw [red,->] (node_10) ..controls (332.5bp,201.2bp) and (297.18bp,178.83bp)  .. (node_16);
  \draw (330.5bp,186.0bp) node {$2$};
  \draw [red,->] (node_7) ..controls (425.5bp,59.112bp) and (471.58bp,35.903bp)  .. (node_9);
  \draw (480.5bp,44.0bp) node {$2$};
  \draw [red,->] (node_16) ..controls (217.9bp,130.04bp) and (175.6bp,107.24bp)  .. (node_11);
  \draw (214.5bp,115.0bp) node {$2$};
  \draw [red,->] (node_13) ..controls (217.9bp,59.04bp) and (175.6bp,36.236bp)  .. (node_2);
  \draw (214.5bp,44.0bp) node {$2$};
  \draw [blue,->] (node_10) ..controls (363.0bp,202.44bp) and (363.0bp,183.5bp)  .. (node_4);
  \draw (371.5bp,186.0bp) node {$1$};
  \draw [blue,->] (node_12) ..controls (826.37bp,130.11bp) and (871.11bp,106.9bp)  .. (node_8);
  \draw (879.5bp,115.0bp) node {$1$};
  \draw [red,->] (node_3) ..controls (682.0bp,202.44bp) and (682.0bp,183.5bp)  .. (node_5);
  \draw (690.5bp,186.0bp) node {$2$};
  \draw [red,->] (node_15) ..controls (667.0bp,60.442bp) and (667.0bp,41.496bp)  .. (node_6);
  \draw (675.5bp,44.0bp) node {$2$};
  \draw [red,->] (node_12) ..controls (789.0bp,131.44bp) and (789.0bp,112.5bp)  .. (node_20);
  \draw (797.5bp,115.0bp) node {$2$};
  \draw [red,->] (node_18) ..controls (524.29bp,279.87bp) and (439.75bp,249.8bp)  .. (node_10);
  \draw (492.5bp,257.0bp) node {$2$};
  \draw [blue,->] (node_4) ..controls (369.32bp,131.34bp) and (376.0bp,112.12bp)  .. (node_7);
  \draw (386.5bp,115.0bp) node {$1$};
  \draw [red,->] (node_20) ..controls (789.0bp,60.442bp) and (789.0bp,41.496bp)  .. (node_1);
  \draw (797.5bp,44.0bp) node {$2$};
  \draw [red,->] (node_11) ..controls (97.894bp,59.299bp) and (65.591bp,37.231bp)  .. (node_0);
  \draw (96.5bp,44.0bp) node {$2$};
  \draw [blue,->] (node_8) ..controls (950.11bp,59.299bp) and (982.41bp,37.231bp)  .. (node_14);
  \draw (992.5bp,44.0bp) node {$1$};
  \draw [blue,->] (node_5) ..controls (711.86bp,130.25bp) and (746.3bp,108.03bp)  .. (node_20);
  \draw (756.5bp,115.0bp) node {$1$};
  \draw [red,->] (node_4) ..controls (332.5bp,130.2bp) and (297.18bp,107.83bp)  .. (node_13);
  \draw (330.5bp,115.0bp) node {$2$};
  \draw [red,->] (node_8) ..controls (922.0bp,60.442bp) and (922.0bp,41.496bp)  .. (node_21);
  \draw (930.5bp,44.0bp) node {$2$};
  \draw [blue,->] (node_16) ..controls (254.0bp,131.44bp) and (254.0bp,112.5bp)  .. (node_13);
  \draw (262.5bp,115.0bp) node {$1$};
\end{tikzpicture}
\]
\caption{The first four levels of $B(4\rho)$ of type $A_2$.}
\label{fig:A2_4rho}
\end{figure}
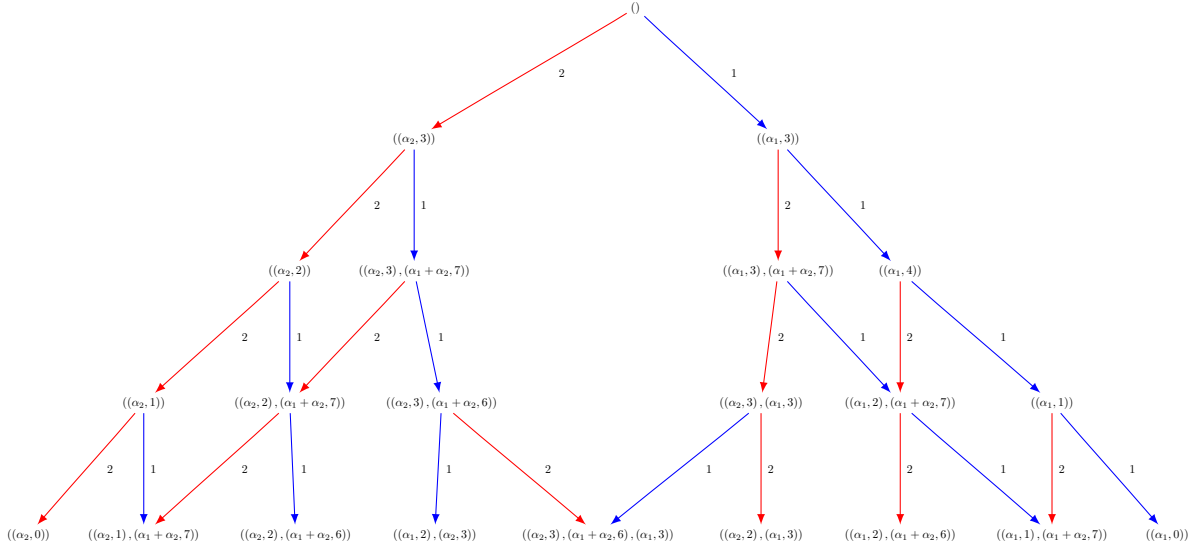

\begin{ex}
We consider the map $S^{\pr}_{4\rho} \colon \Al(\infty) \to T_{-4\rho} \otimes \Al(4\rho)$ in type $A_2$ given by
\begin{align*}
(\alpha_a, j) & \mapsto t_{-4\rho} \otimes (\alpha_a, j+4) \hspace{40pt} (a \in \{1, 2\}),
\\( \alpha_1 + \alpha_2, \ell) & \mapsto t_{-4\rho} \otimes (\alpha_1 + \alpha_2, \ell + 8).
\end{align*}
In particular, compare the elements of Figure~\ref{fig:A2_example} with the corresponding element in Figure~\ref{fig:A2_4rho}.
\end{ex}

\begin{lemma}
\label{lemma:shift_morphism_inf}
The maps $S^{\pr}_{k\rho}$ and $S^{\inc}_{-k\rho}$ are crystal surjections and embeddings, respectively.
\end{lemma}

\begin{proof}
This is similar argument the proof of Lemma~\ref{lemma:shift_morphism}.
\end{proof}

\begin{lemma}
\label{lemma:alcove_direct_limit}
The family $\{T_{-k\rho} \otimes \Al(k\rho)\}_{k=0}^{\infty}$ forms a directed system with inclusion maps $\psi_{k',k} = S_{(k' - k)\rho}$, for all $k' > k$.
Moreover, the map
\[
S \colon \varinjlim_{k \in \ZZ_{\geq 0}} T_{-k\rho} \otimes \Al(k\rho) \to \Al(\infty)
\]
given by $S^{\inc}_{-k\rho} \circ \psi^{(k)}$, where $\psi^{(k)} \colon \varinjlim_{k \in \ZZ_{\geq 0}} \Al(k\rho) \to T_{-k\rho} \otimes \Al(k\rho)$ is the natural restriction, is a crystal isomorphism.
\end{lemma}

\begin{proof}
First, $\{T_{-k\rho} \otimes \Al(k\rho)\}_{k=0}^{\infty}$ is a directed system by Lemma~\ref{lemma:shift_morphism} and clearly $S_{k\rho} \circ S_{k'\rho} = S_{(k+k')\rho}$.
Next, we note that Lemma~\ref{lemma:shift_morphism_inf} implies that $S$ is well-defined.
For all $J \in \Al(\infty)$, we have $S^{\pr}_{k\rho}(J) \in T_{-k\rho} \otimes \Al(k\rho)$ for all $k \geq \max_{j \in J} -\ell_j$, which is well-defined since $J$ is a finite set.
Hence, $S$ is invertible, and the claim follows.
\end{proof}

We note that $S$ does not give an equality between the direct limit and $\Al(\infty)$ as the direct limit is an quotient of alcove paths that \emph{start} at the fundamental alcove and alcove paths in $\Al(\infty)$ do not have a well-defined starting point.

\begin{thm}
\label{thm:infinite_isomorphism}
Let $\g$ be of symmetrizable type. Then we have
\[
\Al(\infty) \iso B(\infty).
\]
\end{thm}

\begin{proof}
We will define a map $\Psi \colon \Al(\infty) \to B(\infty)$ as follows.
Fix some $b \in B(\infty)$.
Recall the natural projection $p_{\lambda} \colon B(\infty) \to T_{-\lambda} \otimes B(\lambda)$ and inclusion $i_{\lambda} \colon T_{-\lambda} \otimes B(\lambda) \to B(\infty)$ maps from Section~\ref{sec:crystal_bg}.
Let $k$ be such that $p_{k\rho}(b) \neq 0$.
From Theorem~\ref{thm:isomorphic_highest_weight}, we have a (canonical\footnote{Recall that $B(\lambda)$ and $B(\infty)$ admit no non-trivial automorphisms.}) isomorphism $\Phi \colon \Al(k\rho) \to B(k\rho)$.
Thus, we define $\Psi(b)$ by the composition
\[
\Al(\infty) \xrightarrow[\hspace{25pt}]{S^{\pr}_{-k\rho}} T_{-k\rho} \otimes \Al(k\rho) \xrightarrow[\hspace{25pt}]{\Phi} T_{-k\rho} \otimes B(k\rho) \xrightarrow[\hspace{25pt}]{i_{k\rho}} B(\infty).
\]
Note that Lemma~\ref{lemma:shift_morphism_inf} and Lemma~\ref{lemma:alcove_direct_limit} states that this is independent of the choice of $k$ and well-defined.
Additionally, the (local) inverse of $\Psi$ is given by the composition
\[
B(\infty) \xrightarrow[\hspace{25pt}]{p_{k\rho}} T_{-k\rho} \otimes B(k\rho) \xrightarrow[\hspace{25pt}]{\Phi^{-1}} T_{-k\rho} \otimes \Al(k\rho) \xrightarrow[\hspace{25pt}]{S^{\inc}_{k\rho}} \Al(\infty).
\]
Therefore, the map $\Psi$ is an isomorphism as desired.
\end{proof}

Our construction, geometrically speaking, is to extend the alcove walk in the anti-dominant chamber to infinity, but to shift the origin so that it is at the end of the path.
Note that this differs from the construction of $\Al^{\vee}(\lambda)$ in Section~\ref{sec:alcove_reversed}, where the direction of the path is also reversed.


\begin{remark}
If $A_{ij} A_{ji} < 4$ for all $i \neq j \in I$ (\textit{i.e.}, the restriction to any rank 2 Levi subalgebra is of finite type), then we could use the Yang--Baxter moves of~\cite{Lenart07} to construct the directed system $\{T_{-\lambda} \otimes \Al(\lambda)\}_{\lambda \in P^+}$. However, it would be interesting to construct this for general symmetrizable types as it could allow one to determine the subset of $\Al(\infty)$ that corresponds to $\Al(\lambda)$ and generalize the model for any $\infty$-chain of the lex $\rho$-chain.
\end{remark}

We also have the following for the dual alcove path model.

\begin{cor}
Let $\g$ be a symmetrizable Kac--Moody algebra. Then we have
\[
\Al^{\vee}(\infty) \iso B(\infty)^{\vee}.
\]
\end{cor}

\begin{proof}
This follows from Theorem~\ref{thm:infinite_isomorphism} and Proposition~\ref{prop:reversing_infinite_alcoves}.
\end{proof}

\section{Continuous limit of infinite alcove walks}
\label{sec:continuous_limit infinite}

We will show that we can extend the dual crystal isomorphism $\varpi_{\lambda} \colon \Al(\lambda) \to \Pi(\lambda)^{\vee}$ to a dual crystal isomorphism $\varpi_{\infty} \colon \Al(\infty) \to \Pi(\infty)^{\vee}$.
We first need to construct a model $\Pi^{\vee}(\infty)$ using somewhat different paths such that $\Pi^{\vee}(\infty) \iso \Pi(\infty)^{\vee}$.

From Theorem~\ref{thm:direct_limit_construction} and the tensor product rule, for any sequence $(a_j \in I)_{j=1}^N$, there exists a $K$ such that
\[
f_{a_1} \cdots f_{a_N} u_{\infty} \mapsto t_{-k\rho} \otimes u_{\infty} \otimes (f_{a_1} \cdots f_{a_N} u_{k\rho}) \in T_{-k\rho} \otimes B(\infty) \otimes B(k\rho)
\]
for all $k > K$.
In terms of the Littelmann path model, there is some $k$ such that
\[
f_{a_1} \cdots f_{a_N} \pi_{\infty} = (f_{a_1} \cdots f_{a_N} \pi_{k\rho}) \ast \pi_{\infty}.
\]

Define $\Pi^{\vee}(\infty)$ be the set of paths (up to $\sim$) $\xi \colon (-\infty, 0] \to \hh_{\RR}^{\ast}$ in the closure of $\xi_{\infty}(t) = t\rho$ under the crystal operators given in Section~\ref{sec:littelmann_path} except with $m_{i,\pi} = \max\{ H_{i,\pi}(t) \mid t \in (-\infty, 0] \}$, interchanging $e_i$ and $f_i$, and $\wt(\xi) = -\xi(0)$.
We can also make this construction geometrically by considering the paths as in the one-point compactification of $\hh_{\RR}$ and performing the usual path reversal and shifting the endpoint.
Indeed, $\Pi^{\vee}(\infty)$ is a subset of all paths $\xi \colon (-\infty, 0] \to \hh_{\RR}$ such that there exists a $T$ where $\xi'(t) = \rho$ for all $t \leq T$.
However, unlike for paths with finite length and $\Pi(\infty)$, we have $\xi(0) = 0$ if and only if $\xi = \xi_{\infty}$.
We also have the following analog of Proposition~\ref{prop:dual_littelmann_path}.

\begin{prop}
\label{prop:dual_path_isomorphism}
We have $\Pi^{\vee}(\infty) \iso \Pi(\infty)^{\vee}$, where the dual crystal isomorphism is given by
\[
\xi^{\vee}(t) = \xi(-t) - \xi(0).
\]
\end{prop}

\begin{proof}
This follows immediately from the definition of $e_i$ and $f_i$ and that $\xi_{\infty}^{\vee} = \pi_{\infty}$.
\end{proof}

\begin{remark}
The set $\Pi^{\vee}(\infty)$ should not be considered as $\Pi(-\infty) = \varinjlim_{k \in \ZZ_{\geq 0}} \Pi(-k\rho)$ as the latter consists of paths $\pi \colon [0, \infty) \to \hh^*_{\RR}$ and must start at the origin.
Additionally, note that $\Pi(-\infty)$ is isomorphic to $\Pi(\infty)^{\vee}$ by Proposition~\ref{prop:dual_littelmann_path} applied to the direct limit (or by restricting to $[0, T)$, where $T$ is minimal such that $\pi'(t) = \rho$ for all $t > T$ and then appending $\pi_{-\infty}(t) = -\rho t$).
However, in order to obtain the continuous limit of $\Al(\infty)$, we require $\Pi^{\vee}(\infty)$ as we do not have a (fixed) starting point for alcove walks in $\Al(\infty)$.
\end{remark}

Therefore, we define our desired dual crystal isomorphism $\varpi_{\infty}$ as the following composition
\[
\begin{array}{ccccccc}
\Al(\infty) & \to & T_{-k\rho} \otimes \Al(\infty) \otimes \Al(k\rho) & \to & \Pi(-k\rho) \otimes \Pi^{\vee}(\infty) \otimes T_{k\rho} & \to & \Pi^{\vee}(\infty)
\\ J & \mapsto & t_{-k\rho} \otimes \emptyset \otimes S_{k\rho}(J) & \mapsto & \varpi_{-k\rho}\bigl(S_{k\rho}(J)\bigr) \otimes \xi_{\infty} \otimes t_{k\rho}& \mapsto & \xi_{\infty} \ast \varpi_{-k\rho}\bigl(S_{k\rho}(J)\bigr)
\end{array}
\]
for some $k \gg 1$ depending on the element $J$.
Hence, by Theorem~\ref{thm:alcove_to_path} we have the following.

\begin{thm}
\label{thm:infinity_alcove_to_dual_path}
Let $\g$ be of symmetrizable type. Then the map
\[
\varpi_{\infty} \colon \Al(\infty) \to \Pi^{\vee}(\infty)
\]
defined above is a dual crystal isomorphism. Moreover, the dual crystal isomorphism is given explicitly by the same description as $\varpi_{\lambda}$ given in Section~\ref{sec:continuous_limit}.
\end{thm}

See Figure~\ref{fig:al_geometry_path} for an example.
We can also directly describe an isomorphism $\Al(\infty) \iso \Pi(\infty)$ by combining the results of Theorem~\ref{thm:infinity_alcove_to_dual_path} and Proposition~\ref{prop:dual_path_isomorphism}.
Furthermore, we have a dual version of Theorem~\ref{thm:infinity_alcove_to_dual_path}.

\begin{thm}
\label{thm:dual_infinity_alcove_to_path}
Let $\g$ be of symmetrizable type. Then the map
\[
\begin{array}{ccccccc}
\Al^{\vee}(\infty) & \to & T_{-k\rho} \otimes \Al^{\vee}(\infty) \otimes \Al^{\vee}(k\rho) & \to & T_{-k\rho} \otimes \Pi(\infty) \otimes \Pi(k\rho) & \to & \Pi(\infty),
\\ J & \mapsto & t_{-k\rho} \otimes \emptyset \otimes S_{k\rho}(J) & \mapsto & t_{-k\rho} \otimes \pi_{\infty} \otimes \varpi_{k\rho}^{\vee}\bigl(S_{k\rho}(J)\bigr) & \mapsto & \varpi_{k\rho}^{\vee}\bigl(S_{k\rho}(J)\bigr) \ast \pi_{\infty},
\end{array}
\]
where $k \gg 1$ depends on the element $J$, is a dual crystal isomorphism.
\end{thm}

\begin{proof}
The proof is similar to Theorem~\ref{thm:infinity_alcove_to_dual_path}, but using Proposition~\ref{prop:dual_littelmann_path} in conjunction with Theorem~\ref{thm:alcove_to_path}.
\end{proof}

See Figure~\ref{fig:al_dual_geometry_path} for an example.
Alternatively this follows from taking the contragredient dual at each step of $\varpi_{\infty}$.

\begin{figure}[t]
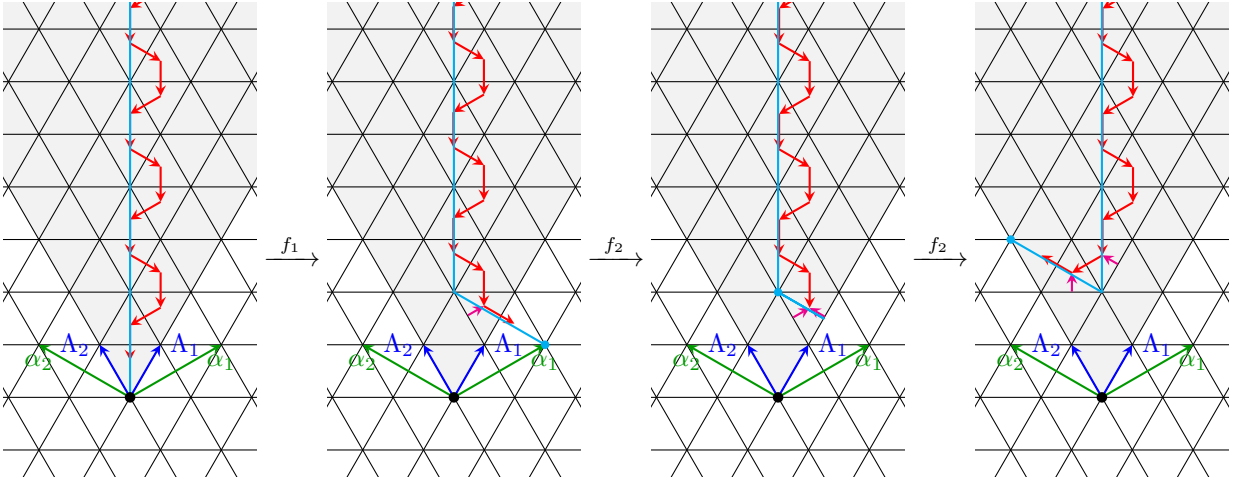

\[
\alcoveex{
  \path[thick, red,<-]
  (0,-1.3) edge +(90:\len)
  ++(90:\len) edge +(30:\len)
  ++(30:\len) edge +(90:\len)
  ++(90:\len) edge +(150:\len)
  ++(150:\len) edge +(90:\len) ++(90:\len) edge +(30:\len) ++(30:\len) edge +(90:\len) ++(90:\len) edge +(150:\len)
  ++(150:\len) edge +(90:\len) ++(90:\len) edge +(30:\len) ++(30:\len) edge +(90:\len) ++(90:\len) edge +(150:\len)
  ++(150:\len) edge +(90:\len) ++(90:\len) edge +(30:\len) ++(30:\len) edge +(90:\len) ++(90:\len) edge +(150:\len)
  ++(150:\len) edge +(90:\len) ++(90:\len) edge +(30:\len) ++(30:\len) edge +(90:\len) ++(90:\len) edge +(150:\len);
  \draw[line cap=round,thick,cyan] (0,10) -- (0,-2);
  \fill[cyan] (0,-2) circle (.08);
}
\xrightarrow[\hspace{15pt}]{f_1}
\alcoveex{
  \path[thick, red,<-]
  (1.15,-0.6) edge +(150:\len)
  ++(150:\len) edge[color=magenta] +(210:\len/2+0.015)
   edge +(90:\len)
  ++(90:\len) edge +(150:\len)
  ++(150:\len) edge +(90:\len) ++(90:\len) edge +(30:\len) ++(30:\len) edge +(90:\len) ++(90:\len) edge +(150:\len)
  ++(150:\len) edge +(90:\len) ++(90:\len) edge +(30:\len) ++(30:\len) edge +(90:\len) ++(90:\len) edge +(150:\len)
  ++(150:\len) edge +(90:\len) ++(90:\len) edge +(30:\len) ++(30:\len) edge +(90:\len) ++(90:\len) edge +(150:\len)
  ++(150:\len) edge +(90:\len) ++(90:\len) edge +(30:\len) ++(30:\len) edge +(90:\len) ++(90:\len) edge +(150:\len);
  \draw[line cap=round,join=round,thick,cyan] (0,10) -- (0,0) -- +(-30:2);
  \fill[cyan] (-30:2) circle (.08);
}
\xrightarrow[\hspace{15pt}]{f_2}
\alcoveex{
  \path[thick, red,<-]
  (0.6,-0.3) edge[color=magenta] +(-30:\len/2)
  edge[color=magenta] +(210:\len/2+0.025)
   edge +(90:\len)
  ++(90:\len) edge +(150:\len)
  ++(150:\len) edge +(90:\len) ++(90:\len) edge +(30:\len) ++(30:\len) edge +(90:\len) ++(90:\len) edge +(150:\len)
  ++(150:\len) edge +(90:\len) ++(90:\len) edge +(30:\len) ++(30:\len) edge +(90:\len) ++(90:\len) edge +(150:\len)
  ++(150:\len) edge +(90:\len) ++(90:\len) edge +(30:\len) ++(30:\len) edge +(90:\len) ++(90:\len) edge +(150:\len)
  ++(150:\len) edge +(90:\len) ++(90:\len) edge +(30:\len) ++(30:\len) edge +(90:\len) ++(90:\len) edge +(150:\len);
  \draw[line cap=round,join=round,thick,cyan] (0,10) -- (0,0) -- +(-30:1) -- (0,0);
  \fill[cyan] (0,0) circle (.08);
}
\xrightarrow[\hspace{15pt}]{f_2}
\alcoveex{
  \path[thick, red,<-]
  (-1.15,0.7) edge +(-30:\len)
  ++(-30:\len) edge[color=magenta] +(-90:\len/2+.02)
   edge +(30:\len)
  ++(30:\len) edge[color=magenta] +(-30:\len/2)
   edge +(90:\len)
  ++(90:\len) edge +(30:\len)
  ++(30:\len) edge +(90:\len)
  ++(90:\len) edge +(150:\len)
  ++(150:\len) edge +(90:\len) ++(90:\len) edge +(30:\len) ++(30:\len) edge +(90:\len) ++(90:\len) edge +(150:\len)
  ++(150:\len) edge +(90:\len) ++(90:\len) edge +(30:\len) ++(30:\len) edge +(90:\len) ++(90:\len) edge +(150:\len)
  ++(150:\len) edge +(90:\len) ++(90:\len) edge +(30:\len) ++(30:\len) edge +(90:\len) ++(90:\len) edge +(150:\len);
  \draw[line cap=round,join=round,thick,cyan] (0,10) -- (0,0) -- +(150:2);
  \fill[cyan] (150:2) circle (.08);
}
\]
\caption{The action of a few crystal operators on $\Al(\infty)$ with the corresponding path in $\Pi^{\vee}(\infty)$ under $\varpi_{\infty}$ in type $A_2$ starting with $\emptyset$ and $\xi_{\infty}$ on the left.}
\label{fig:al_geometry_path}
\end{figure}

\begin{figure}[t]
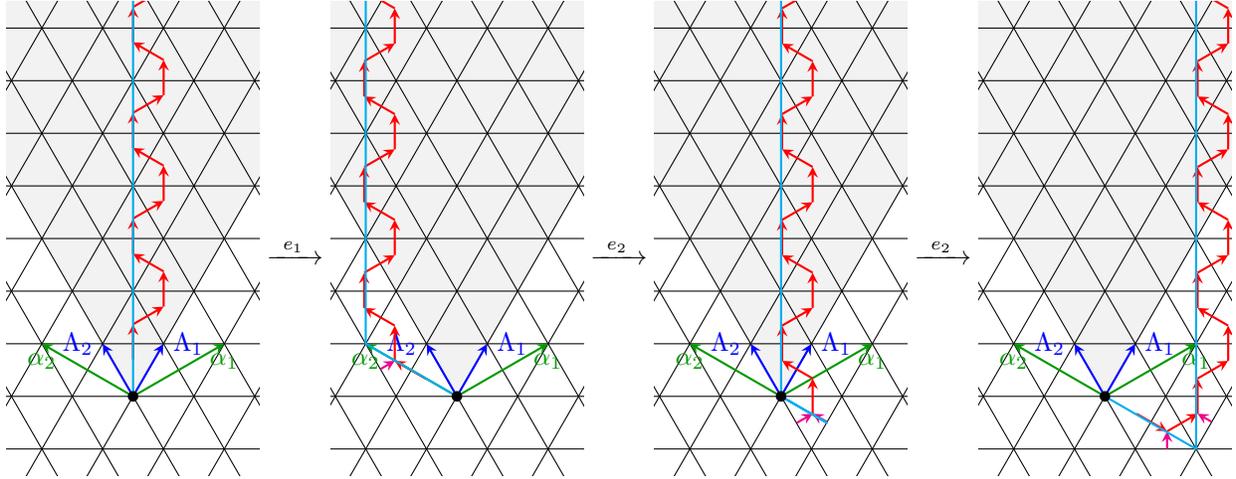

\[
\alcoveex{
  \path[thick, red,->]
  (0,-1.3) edge +(90:\len)
  ++(90:\len) edge +(30:\len)
  ++(30:\len) edge +(90:\len)
  ++(90:\len) edge +(150:\len)
  ++(150:\len) edge +(90:\len) ++(90:\len) edge +(30:\len) ++(30:\len) edge +(90:\len) ++(90:\len) edge +(150:\len)
  ++(150:\len) edge +(90:\len) ++(90:\len) edge +(30:\len) ++(30:\len) edge +(90:\len) ++(90:\len) edge +(150:\len)
  ++(150:\len) edge +(90:\len) ++(90:\len) edge +(30:\len) ++(30:\len) edge +(90:\len) ++(90:\len) edge +(150:\len)
  ++(150:\len) edge +(90:\len) ++(90:\len) edge +(30:\len) ++(30:\len) edge +(90:\len) ++(90:\len) edge +(150:\len);
  \draw[line cap=round,thick,cyan] (0,-2) -- (0,10);
}
\xrightarrow[\hspace{15pt}]{e_1}
\alcoveex{
  \path[thick, red,->]
  (-0.6,-1.66) edge +(150:\len)
  ++(150:\len) edge[color=magenta,<-] +(210:\len/2-0.025)
   edge +(90:\len)
  ++(90:\len) edge +(150:\len)
  ++(150:\len) edge +(90:\len) ++(90:\len) edge +(30:\len) ++(30:\len) edge +(90:\len) ++(90:\len) edge +(150:\len)
  ++(150:\len) edge +(90:\len) ++(90:\len) edge +(30:\len) ++(30:\len) edge +(90:\len) ++(90:\len) edge +(150:\len)
  ++(150:\len) edge +(90:\len) ++(90:\len) edge +(30:\len) ++(30:\len) edge +(90:\len) ++(90:\len) edge +(150:\len)
  ++(150:\len) edge +(90:\len) ++(90:\len) edge +(30:\len) ++(30:\len) edge +(90:\len) ++(90:\len) edge +(150:\len);
  \draw[line cap=round,join=round,thick,cyan] (0,-2) -- ++(150:2) -- +(0,10);
}
\xrightarrow[\hspace{15pt}]{e_2}
\alcoveex{
  \path[thick, red,->]
  (0.6,-2.33) edge[color=magenta,<-] +(-30:\len/2-0.025)
   edge[color=magenta,<-] +(210:\len/2+0.01)
   edge +(90:\len)
  ++(90:\len) edge +(150:\len)
  ++(150:\len) edge +(90:\len) ++(90:\len) edge +(30:\len) ++(30:\len) edge +(90:\len) ++(90:\len) edge +(150:\len)
  ++(150:\len) edge +(90:\len) ++(90:\len) edge +(30:\len) ++(30:\len) edge +(90:\len) ++(90:\len) edge +(150:\len)
  ++(150:\len) edge +(90:\len) ++(90:\len) edge +(30:\len) ++(30:\len) edge +(90:\len) ++(90:\len) edge +(150:\len)
  ++(150:\len) edge +(90:\len) ++(90:\len) edge +(30:\len) ++(30:\len) edge +(90:\len) ++(90:\len) edge +(150:\len);
  \draw[line cap=round,join=round,thick,cyan] (0,-2) -- +(-30:1) -- (0,-2) -- (0,10);
}
\xrightarrow[\hspace{15pt}]{e_2}
\alcoveex{
  \path[thick, red,->]
  (0.6,-2.33) edge +(-30:\len)
  ++(-30:\len) edge[color=magenta,<-] +(-90:\len/2)
   edge +(30:\len)
  ++(30:\len) edge[color=magenta,<-] +(-30:\len/2-0.03)
   edge +(90:\len)
  ++(90:\len) edge +(30:\len) ++(30:\len) edge +(90:\len) ++(90:\len) edge +(150:\len)
  ++(150:\len) edge +(90:\len) ++(90:\len) edge +(30:\len) ++(30:\len) edge +(90:\len) ++(90:\len) edge +(150:\len)
  ++(150:\len) edge +(90:\len) ++(90:\len) edge +(30:\len) ++(30:\len) edge +(90:\len) ++(90:\len) edge +(150:\len)
  ++(150:\len) edge +(90:\len) ++(90:\len) edge +(30:\len) ++(30:\len) edge +(90:\len) ++(90:\len) edge +(150:\len);
  \draw[line cap=round,join=round,thick,cyan] (0,-2) -- ++(-30:2) -- +(0,10);
}
\]
\caption{The action of a few crystal operators on $\Al^{\vee}(\infty)$ with the corresponding path in $\Pi(\infty)$ under the map of Theorem~\ref{thm:infinity_alcove_to_dual_path} in type $A_2$ starting with $\emptyset$ and $\pi_{\infty}$ on the left.}
\label{fig:al_dual_geometry_path}
\end{figure}

\appendix
\section{Calculations using Sage}
\label{sec:sage}

The crystal $\Al(\lambda)$ (resp. $\Al(\infty)$) has been implemented by the first (resp. second) author in Sage~\cite{sage,combinat}.  We conclude with examples.

We construct $\Al(\infty)$ in type $A_3$ and compute the element $b = f_2 f_3 f_1 f_2 f_2 f_3 f_1 f_2 \emptyset$:
\begin{lstlisting}
sage: A = crystals.infinity.AlcovePaths(['A',3])
sage: mg = A.highest_weight_vector()
sage: b = mg.f_string([2,1,3,2,2,1,3,2])
sage: b
((alpha[2], -2), (alpha[2] + alpha[3], -2),
 (alpha[1] + alpha[2], -2), (alpha[1] + alpha[2] + alpha[3], -2))
sage: b.weight()
(-4, -4, 0, 0)
\end{lstlisting}
Next, we construct the projection onto $\Al(2\rho)$ by computing $S_{2\rho}^P(b)$:
\begin{lstlisting}
sage: b.projection()
((alpha[2], 0), (alpha[2] + alpha[3], 2),
 (alpha[1] + alpha[2], 2), (alpha[1] + alpha[2] + alpha[3], 4))
sage: b.to_highest_weight()
[(), [2, 1, 2, 1, 3, 2, 3, 2]]
\end{lstlisting}
Note that $\inner{(k_1 \alpha_1 + k_2 \alpha_2 + k_3 \alpha_3)^{\vee}}{\rho} = k_1 + k_2 + k_3$.
Therefore, compare the result to the corresponding elements in $\Al(3\rho)$ and $\Al(4\rho)$:
\begin{lstlisting}
sage: A = crystals.AlcovePaths(['A',3], [3,3,3])
sage: mg = A.highest_weight_vector()
sage: mg.f_string([2,1,3,2,2,1,3,2])
((alpha[2], 1), (alpha[2] + alpha[3], 4),
 (alpha[1] + alpha[2], 4), (alpha[1] + alpha[2] + alpha[3], 7))
sage: A = crystals.AlcovePaths(['A',3], [4,4,4])
sage: mg = A.highest_weight_vector()
sage: mg.f_string([2,1,3,2,2,1,3,2])
((alpha[2], 2), (alpha[2] + alpha[3], 6),
 (alpha[1] + alpha[2], 6), (alpha[1] + alpha[2] + alpha[3], 10))
\end{lstlisting}

\begin{ex}
\label{ex:concat_admissible}
We give an example showing that simply concatenating the folding positions in $\Al(\Lambda_1) \otimes \Al(\Lambda_1)$ is not equal to $\Al(2\Lambda_1)$ in type $A_3$ (even though they are isomorphic).

\begin{lstlisting}
sage: P = RootSystem(['A',3]).weight_lattice()
sage: La = P.fundamental_weights()
sage: C = crystals.AlcovePaths(2*La[1])
sage: D = crystals.AlcovePaths(La[1])
sage: C.vertices()
[[], [0], [3], [0, 1], [0, 4], [3, 4],
 [0, 1, 2], [0, 1, 5], [0, 4, 5], [3, 4, 5]]
sage: D.vertices()
[[], [0], [0, 1], [0, 1, 2]]
\end{lstlisting}
In particular, note that for the folding positions $\{0, 4, 5\} \in \Al(2\Lambda_1)$, if we consider this as a concatenation, then $\{1,2\} \notin \Al(\Lambda_1)$.
\end{ex}

\section{Alcove model: Crystal operators}
\label{sec:ame}
In this section we show that our description of crystal operators
in Section \ref{section:alcove_model} is equivalent to the one given in \cite{LP08}.

Let $J = \{ j_1 < j_2 < \cdots < j_p \}$.  Recall $\gamma_k$ from Equation~\eqref{eqn:defn_gamma} and the set $I_{\alpha_i}$ from Equation~\eqref{eqn:root_set_I}.  Let 
\[
I_{\alpha_i} = \{ i_1 < i_2 < \cdots < i_N \}\; \mbox{ and } \;
\widehat{I}_{\alpha_i} = I_{\alpha_i} \cup \{\infty\}.
\]
Let $\gamma_{\infty} = r_{j_1}r_{j_2}\cdots r_{j_p}(\rho)$, and define
\[
\varsigma_i := 
\begin{cases}
    1  & \mbox{ if } i \notin J,\\
    -1 & \mbox{ if } i \in J.
\end{cases}
\]
Crystal operators are defined in terms of the 
piecewise linear function $g_{\alpha_i} \colon [0, N+\frac{1}{2}] \to
\RR$ given by
\begin{equation*}
g_{\alpha_i}(0) = -\frac{1}{2},
\qquad\qquad
\frac{dg_{\alpha_i}}{dx}(x) = \begin{cases}
\sgn(\gamma_{i_j}) & \text{if } x \in (j-1, j-\frac{1}{2}),\ j = 1, \dotsc, N,\\
\varsigma_{i_j} \sgn(\gamma_{i_j}) & \text{if } x \in (j-\frac{1}{2}, j),\ j = 1, \dotsc, N, \\
\sgn(\langle\gamma_{\infty}, \alpha_i^{\vee}\rangle) & \text{if } x
\in (N, N+\frac{1}{2}).
\end{cases}
\end{equation*}
The graph $g_{\alpha_i}$ is used to define crystal operators in the
alcove model. Let  
\begin{align*}
\sigma_j &:= \bigl( \sgn(\gamma_{i_j}), \varsigma_{i_j}\sgn(\gamma_{i_j}) \bigr),
\\ \sigma_{N+1} & := \sgn\bigl(\langle\gamma_{\infty}, \alpha_i^{\vee}\rangle \bigr),
\end{align*}
where $1 \leq j \leq N$.
We note the following two conditions from \cite{LP08}:
\begin{itemize}
\item[(C1)] $\sigma_j \in {(1,1), (1,-1), (-1,-1)}$ for $1 \le j \le N$,
\item[(C2)] $\sigma_j = (1,1)$ implies $\sigma_{j+1} \in \{(1,1), (1,-1), 1\}$.
\end{itemize}

In the language of Section~\ref{section:alcove_model}, we identify $(1,1)$ with the symbol $+$
and $(-1,-1)$ with the symbol $-$.  We identify $(1,-1)$ with the symbol
$\pm$ and note that if $\sigma_j = (1,-1)$, then $i_j \in J$. 
Finally identify $\sigma_{N+1} = 1$ with $+$ and $\sigma_{N+1} = -1$ 
with $-$.  Condition (C1) says that we can describe $g_{\alpha_i}$ as
a word in the alphabet $\{ +, -, \pm \}$.  Condition (C2) says that
the transition from $+$ to $-$ must pass through $\pm$.

We now recall the definition of $f_i$.  Let $M$ be the maximum of $g_{{\alpha}_i}$.  
Let $\h{i_j} = g_{\alpha_i}(j-\frac{1}{2})$ and $\h{\infty} = g_{\alpha_i}(N+\frac{1}{2})$.  Let $\mu$ be the minimum index in $\widehat{I}_{{\alpha}_i}$ for which we have $\h{\mu}=M$.  
Then $\mu\in J$ or $\mu = \infty$.
If $M>0$, then $\mu$ has a predecessor $k$ in
$I_{{\alpha}_i}$, with $k\not\in J$.  Define
\begin{equation*}
	\label{eqn:rootF} 
f_i J := 
	\begin{cases}
		(J \setminus \left\{ \mu \right\}) \cup \{ k \} & \text{if } M > 0,\\
				0 & \text{otherwise.}
	\end{cases}
\end{equation*}
We use the convention that $J\backslash \left\{ \infty \right\}= J \cup \left\{ \infty \right\} = J$. 
Observe that after canceling out $-+$ terms as in Section
\ref{section:alcove_model} the rightmost remaining $+$ corresponds to
$k$ and the $\pm$ ( $+$ if $\mu = \infty$) term immediately following corresponds to $\mu$.   
This follows from conditions (C1) and (C2).

We now recall the definition of $e_i$. If  $M >
\h{ \infty }$, let $k$ be the maximum index 
in $I_{{\alpha}_i}$ for which we have $\h{k} = M$,
then $k \in J$ and $k$ has a successor $\mu$ in $\widehat{I}_{{\alpha}_i}$ with $\mu
\not \in J$.
Define 
\begin{equation*}
	\label{eqn:rootE}
e_i J := 
	\begin{cases}
		(J \setminus \left\{ k \right\}) \cup \{ \mu \} & \text{if }
		M > \h{ \infty },\\ 
				0 & \text{otherwise.}
	\end{cases}
\end{equation*}
Here it is also the case by (C1) and (C2) that the left most $-$, which exists
if $M > h_{ \infty }^J$,  
corresponds to $\mu$ and the immediately preceding $\pm$ corresponds to $k$.



\section*{Acknowledgements}

The authors thank Cristian Lenart for many helpful discussions.
The authors thank Ben Salisbury for comments on an early draft of this manuscript.
This work benefited from computations, as well as created Figure~\ref{fig:A2_example} and Figure~\ref{fig:A2_4rho}, using {\sc SageMath}~\cite{combinat,sage}.
We thank the anonymous referee for comments and suggestions.

\bibliographystyle{alpha}
\bibliography{paths}{}
\end{document}